\documentclass{louloupart}
\usepackage{pinlabel}
\usepackage{graphicx}
\usepackage[all]{xy}

\title[Endomorphisms of Artin groups of type D]{Endomorphisms of Artin groups of type D}
\author[F Castel]{Fabrice Castel}
\givenname{Fabrice}
\surname{Castel}
\address{Fabrice Castel, IMB, UMR 5584, CNRS, Universit\'e de Bourgogne, 21000 Dijon, France
\vskip 3 pt
Luis Paris, IMB, UMR 5584, CNRS, Universit\'e de Bourgogne, 21000 Dijon, France
}
\email{fabrice.castel@gmail.com}
\author[L Paris]{Luis Paris}
\givenname{Luis}
\surname{Paris}
\email{lparis@u-bourgogne.fr}

\newtheorem{thm}{Theorem}[section]
\newtheorem{lem}[thm]{Lemma}
\newtheorem{prop}[thm]{Proposition}
\newtheorem{corl}[thm]{Corollary}

\theoremstyle{definition}

\newtheorem*{rem}{Remark}

\newtheorem*{acknow}{Acknowledgments}

\numberwithin{equation}{section}

\makeatletter
\renewcommand{\thefigure}{\ifnum \c@section>\z@ \thesection.\fi
 \@arabic\c@figure}
\@addtoreset{figure}{section}
\makeatother

\begin{document}

\def\N{\mathbb N} \def\Aut{{\rm Aut}} \def\ad{{\rm ad}}
\def\Z{\mathbb Z}\def\id{{\rm id}} \def\Ker{{\rm Ker}}
\def\Inn{{\rm Inn}} \def\Out{{\rm Out}} \def\PP{\mathcal P}
\def\Homeo{{\rm Homeo}} \def\MM{\mathcal M} \def\CC{\mathcal C}
\def\AA{\mathcal A} \def\S{\mathbb S} \def\SS{\mathcal S}
\def\Im{{\rm Im}} \def\D{\mathbb D}


\begin{abstract}
In this paper we determine a classification of the endomorphisms of the Artin group $A [D_n]$ of type $D_n$ for $n\ge 6$.
In particular we determine its automorphism group and its outer automorphism group.
We also determine a classification of the homomorphisms from $A[D_n]$ to the Artin group $A [A_{n-1}]$ of type $A_{n-1}$ and a classification of the homomorphisms from $A[A_{n-1}]$ to $A[D_n]$ for $n\ge 6$.
We show that any endomorphism of the quotient $A [D_n] / Z (A [D_n])$ lifts to an endomorphism of $A [D_n]$ for $n \ge 4$.
We deduce  a classification of the endomorphisms of $A [D_n] / Z (A [D_n])$, we determine the automorphism and outer automorphism groups of $A [D_n] / Z (A [D_n])$, and we show that $A [D_n] / Z (A [D_n])$ is co-Hopfian, for $n \ge 6$.
The results are algebraic in nature but the proofs are based on topological arguments (curves on surfaces and mapping class groups).

\smallskip\noindent
{\bf AMS Subject Classification\ \ } 
Primary: 20F36.
Secondary: 57K20.

\smallskip\noindent
{\bf Keywords\ \ } 
Artin groups of type $D$, endomorphisms, automorphisms, mapping class groups.

\end{abstract}

\maketitle


\section{Introduction}\label{sec1}

Let $S$ be a finite set.
A \emph{Coxeter matrix} over $S$ is a square matrix $M=(m_{s,t})_{s,t\in S}$ indexed by the elements of $S$, with coefficients in $\N\cup\{\infty\}$, such that $m_{s,s}=1$ for all $s\in S$ and $m_{s,t}=m_{t,s}\ge 2$ for all $s,t\in S$, $s\neq t$.
Such a matrix is usually represented by a labeled graph $\Gamma$, called a \emph{Coxeter graph}, defined as follows.
The set of vertices of $\Gamma$ is $S$.
Two vertices $s,t\in S$ are connected by an edge if $m_{s,t}\ge 3$, and this edge is labeled with $m_{s,t}$ if $m_{s,t}\ge4$.

If $a,b$ are two letters and $m$ is an integer $\ge2$, then we denote by $\Pi(a,b,m)$ the word $aba\cdots$ of length $m$.
In other words $\Pi(a,b,m)=(ab)^{\frac{m}{2}}$ if $m$ is even and $\Pi(a,b,m)=(ab)^{\frac{m-1}{2}}a$ if $m$ is odd.
Let $\Gamma$ be a Coxeter graph and let $M=(m_{s,t})_{s,t\in S}$ be its Coxeter matrix.
With $\Gamma$ we associate a group $A[\Gamma]$, called the \emph{Artin group} of $\Gamma$, defined by the following presentation.
\[
A[\Gamma]=\langle S\mid\Pi(s,t,m_{s,t})=\Pi(t,s,m_{s,t})\text{ for }s,t\in S\,,\ s\neq t\,,\ m_{s,t}\neq\infty\rangle\,.
\]
The \emph{Coxeter group} of $\Gamma$, denoted $W[\Gamma]$, is the quotient of $A[\Gamma]$ by the relations $s^2=1$, $s\in S$.

Despite the popularity of Artin groups little is known on their automorphisms and even less on their endomorphisms.
The most emblematic cases are the braid groups and the right-angled Artin groups.
Recall that the \emph{braid group} on $n+1$ strands is the Artin group $A[A_n]$ where $A_n$ is the Coxeter graph depicted in Figure \ref{fig1_1}, and an Artin group $A[\Gamma]$ is called a \emph{right-angled Artin group} if $m_{s,t}\in\{2,\infty\}$ for all $s,t\in S$, $s\neq t$.
The automorphism group of $A[A_n]$ was determined by Dyer--Grossman \cite{DyeGro1} and the set of its endomorphisms by Castel \cite{Caste2} for $n\ge 5$, by Chen--Kordek--Margalit \cite{ChKoMa1} for $n\ge 3$ and by Orevkov \cite{Orevk1} for $n \ge 2$ (see also Bell--Margalit \cite{BelMar1} and Kordek--Margalit \cite{KorMar1}).
On the other hand there are many articles studying automorphism groups of right-angled Artin groups (see Charney--Vogtmann \cite{ChaVog1,ChaVog2}, Day \cite{Day1,Day2}, Laurence \cite{Laure1} and Bregman--Charney--Vogtmann \cite{BrChVo1} for example), but almost nothing is known on endomorphisms of these groups.

\begin{figure}[ht!]
\begin{center}
\includegraphics[width=3.2cm]{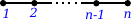}
\caption{Coxeter graph $A_n$}\label{fig1_1}
\end{center}
\end{figure}

Apart from these two families little is known on automorphisms of Artin groups.
The automorphism groups of two generators Artin groups were determined in Gilbert--Howie--Metaftsis--Raptis \cite{GiHoMeRa1}, the automorphism groups of the Artin groups of type $B_n$, $\tilde A_n$ and $\tilde C_n$ were determined in Charney--Crisp \cite{ChaCri1}, 
the automorphisms groups of some $2$-dimensional Artin groups were determined in Crisp \cite{Crisp1}
and in An--Cho \cite{AnCho1}, the automorphism groups of large-type free-of-infinity Artin groups were determined in Vaskou \cite{Vasko1}, and the automorphism group of $A[D_4]$ was determined in Soroko \cite{Sorok1}.
On the other hand, as far as we know the set of endomorphisms of an Artin group is not determined for any Artin group except for those of type $A_n$.

Recall that an Artin group $A[\Gamma]$ is of \emph{spherical type} if $W[\Gamma]$ is finite.
The study of spherical-type Artin groups began in the early 1970s with works by Brieskorn \cite{Bries2,Bries1}, Brieskorn--Saito \cite{BriSai1} and Deligne \cite{Delig1}, works that marked in a way the beginning of the theory of Artin groups.
This family and that of right-angled Artin groups are the two most studied and best understood families of Artin groups and, obviously, any question on Artin groups first arises for Artin groups of spherical type and for right-angled Artin groups.
Here we are interested in Artin groups of spherical type and more particularly in those of type $D_n$.

An Artin group $A[\Gamma]$ is called \emph{irreducible} if $\Gamma$ is connected.
If $\Gamma_1,\dots,\Gamma_\ell$ are the connected components of $\Gamma$, then $A[\Gamma]=A[\Gamma_1]\times\cdots\times A[\Gamma_\ell]$ and $W[\Gamma]=W[\Gamma_1]\times\cdots\times W[ \Gamma_\ell]$.
In particular $A[\Gamma]$ is of spherical type if and only if $A[\Gamma_i]$ is of spherical type for all $i\in\{1,\dots,\ell\}$.
So, to classify Artin groups of spherical type it suffices to classify those which are irreducible.
Finite irreducible Coxeter groups and hence irreducible Artin groups of spherical type were classified by Coxeter \cite{Coxet2,Coxet1}.
There are four infinite families, $A_n$ ($n\ge1$), $B_n$ ($n\ge 2$), $D_n$ ($n\ge 4$) and $I_2(m)$ ($ m\ge 5$), and six ``sporadic'' groups, $E_6$, $E_7$, $E_8$, $F_4$, $H_3$ and $H_4$.
As mentioned above, the automorphism group of $A[\Gamma]$ for $\Gamma$ of type $A_n$ ($n\ge 1$), $B_n$ ($n\ge 2$) and $I_2 (m)$ ($m\ge 5$) is known.
The next step is therefore to understand the automorphism group of $A[D_n]$ for $n\ge 5$ (the case $\Gamma=D_4$ is known by Soroko \cite{Sorok1}).
The Coxeter graph $D_n$ is illustrated in Figure \ref{fig1_2}.

\begin{figure}[ht!]
\begin{center}
\includegraphics[width=4cm]{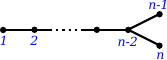}
\caption{Coxeter graph $D_n$}\label{fig1_2}
\end{center}
\end{figure}

In this paper we determine a complete and precise classification of the endomorphisms of $A[D_n]$ for $n\ge 6$ (see Theorem \ref{thm2_3}).
In particular we determine the automorphism group and the outer automorphism group of $A[D_n]$ for $n\ge6$ (see Corollary \ref{corl2_6}).
We also determine a complete and precise classification of the homomorphisms from $A[D_n]$ to $A[A_{n-1}]$ (see Theorem \ref{thm2_1}) and a complete and precise classification of the homomorphisms from $A[A_{n-1}]$ to $A[D_n]$ (see Theorem \ref{thm2_2}).
Note that all these results were announced but not proved in Castel \cite{Caste1}; 
actually the proofs turn out to be much more difficult than the first author thought when he announced them.
Note also that our techniques cannot be used to treat the cases $n=4$ and $n=5$.
In particular we do not know how to determine $\Aut(A[D_5])$.

From our main result we deduce a classification of the endomorphisms of $A [D_n] / Z (A [D_n])$ for $n \ge 6$, where $Z (A [D_n])$ denotes the center of $A [D_n]$ (see Theorem \ref{thmA2}). 
Then we determine the automorphism group and the outer automorphism group of $A [D_n] / Z (A [D_n])$ (see Corollary \ref{corlA3}), and we show that $A [D_n] / Z (A [D_n])$ is co-Hopfian (see Corollary \ref{corlA4}). 
These results follow from Theorem \ref{thm2_3} and Proposition \ref{propA1} which states that any endomorphism of $A [D_n] / Z (A [D_n])$ lifts to an endomorphism of $A [D_n]$. 
Such results were previously known for braid groups, that is, Artin groups of type $A_n$ (see Bell--Margalit \cite{BelMar1}). 
Note that the application of our main result to the study of $A [D_n] / Z (A [D_n])$ was not present in an earlier version of the paper.
It was suggested to us by the referee and we thank her/him warmly for that.

A \emph{geometric representation} of an Artin group is a homomorphism from the group to a  mapping class group (see Section \ref{sec3} for more details).
In order to achieve our goals we make a study of a particular geometric representation of $A[D_n]$ previously introduced by Perron--Vannier \cite{PerVan1} with one boundary component replaced by a puncture.
This geometric representation will be the key tool for many of our proofs.
Overall, although the results of the paper are algebraic in nature, the proofs are mostly based on topological arguments (on curves on surfaces and mapping class groups).

The paper is organized as follows.
In Section \ref{sec2} we give the main definitions and precise statements of the main results.
Section \ref{sec3} is dedicated to the study of some geometric representations of Artin groups of type $A_n$ and type $D_n$.
In Section \ref{sec4} we determine the homomorphisms from $A[D_n]$ to $A[A_{n-1}]$, in Section \ref{sec5} we determine the homomorphisms from $A[A_{n-1}]$ to $ A[D_n]$, and in Section \ref{sec6} we determine the endomorphisms of $A[D_n]$.
In Section \ref{sec7} we determine the endomorphisms of $A [D_n] / Z (A [D_n])$.

\begin{acknow}
The authors would like to thank Bruno Cisneros de la Cruz and Juan Gonz\'alez-Meneses for helpful comments and conversations.
They also want to thank the referee for many helpful remarks.
The second author is partially supported by the French project ``AlMaRe'' (ANR-19-CE40-0001-01) of the ANR.
\end{acknow}


\section{Definitions and statements}\label{sec2}

For $n\ge4$ we denote by $s_1,\dots,s_{n-1}$ the standard generators of $A[A_{n-1}]$ numbered as in Figure \ref{fig1_1} and by $t_1,\dots,t_n$ the standard generators of $A[D_n]$ numbered as in Figure \ref{fig1_2}.

Let $\Gamma$ be a Coxeter graph.
For $X\subset S$ we denote by $A_X=A_X[\Gamma]$ the subgroup of $A=A[\Gamma]$ generated by $X$, by $W_X=W_X[\Gamma]$ the subgroup of $W=W[\Gamma]$ generated by $X$, and by $\Gamma_X$ the full subgraph of $\Gamma$ spanned by $X$.
We know from van der Lek \cite{Lek1} that $A_X$ is the Artin group of $\Gamma_X$ and from Bourbaki \cite{Bourb1} that $W_X$ is the Coxeter group of $\Gamma_X$.
A subgroup of the form $A_X$ is called a \emph{standard parabolic subgroup} of $A$ and a subgroup of the form $W_X$ is called a \emph{standard parabolic subgroup} of $W$.

For $w\in W$ we denote by $\lg(w)$ the word length of $w$ with respect to $S$.
A \emph{reduced expression} for $w$ is an expression $w=s_1s_2\cdots s_\ell$ of minimal length, that is, such that $\ell=\lg(w)$.
Let $\omega:A\to W$ be the natural epimorphism which sends $s$ to $s$ for all $s\in S$.
This epimorphism has a natural set-section $\tau:W\to A$ defined as follows.
Let $w\in W$ and let $w=s_1s_2\cdots s_\ell$ be a reduced expression for $w$.
Then $\tau(w)=s_1s_2\cdots s_\ell\in A$.
We know from Tits \cite{Tits1} that the definition of $\tau(w)$ does not depend on the choice of its reduced expression.

Assume $\Gamma$ is of spherical type.
Then $W$ has a unique element of maximal length, denoted $w_S$, which satisfies $w_S^2=1$ and $w_SSw_S=S$.
The \emph{Garside element} of $A$ is defined to be $\Delta=\Delta[\Gamma]=\tau(w_S)$.
We know that $\Delta S\Delta^{-1}=S$ and, if $\Gamma$ is connected, then the center $Z(A)$ of $A$ is an infinite cyclic group generated by either $\Delta$ or $\Delta^2$ (see Brieskorn--Saito \cite{BriSai1}).
For $X\subset S$ we denote by $w_X$ the element of maximal length in $W_X$ and by $\Delta_X=\Delta_X[\Gamma]=\tau(w_X)$ the Garside element of $A_X$.

If $\Gamma=A_{n-1}$, then
\[
\Delta=(s_{n-1}\cdots s_1)(s_{n-1}\cdots s_2)\cdots(s_{n-1}s_{n-2})s_{n-1}\,,
\]
$\Delta s_i\Delta^{-1}=s_{n-i}$ for all $1\le i\le n-1$ and $Z(A)$ is generated by $\Delta^2$.
If $\Gamma=D_n$, then
\[
\Delta=(t_1\cdots t_{n-2}t_{n-1}t_nt_{n-2}\cdots t_1)(t_2\cdots t_{n-2}t_{n-1}t_nt_{n-2 }\cdots t_2)
\cdots(t_{n-2}t_{n-1}t_nt_{n-2})(t_{n-1}t_n)\,.
\]
If $n$ is even, then $\Delta t_i\Delta^{-1}=t_i$ for all $1\le i\le n$ and $Z(A)$ is generated by $\Delta$.
If $n$ is odd, then $\Delta t_i\Delta^{-1}=t_i$ for all $1\le i\le n-2$, $\Delta t_{n-1}\Delta^{-1}=t_n$, $\Delta t_n\Delta^{-1}=t_{n-1}$, and $Z(A)$ is generated by $\Delta^2$.

If $G$ is a group and $g\in G$, then we denote by $\ad_g:G\to G$, $h\mapsto ghg^{-1}$, the conjugation map by $g$.
We say that two homomorphisms $\varphi_1,\varphi_2:G\to H$ are \emph{conjugate} if there exists $h\in H$ such that $\varphi_2=\ad_h\circ\varphi_1$.

A homomorphism $\varphi:G\to H$ is called \emph{abelian} if its image is an abelian subgroup of $H$.
A homomorphism $\varphi:G\to H$ is called \emph{cyclic} if its image is a cyclic subgroup of $H$.
If $G=A[A_{n-1}]$, then $\varphi:A[A_{n-1}]\to H$ is abelian if and only if it is cyclic, if and only if there exists $h\in H$ such that $\varphi(s_i)=h$ for all $1\le i\le n-1$.
Similarly, if $G=A[D_n]$, then $\varphi:A[D_n]\to H$ is abelian if and only if it is cyclic, if and only if there exists $h\in H$ such that $\varphi(t_i)=h$ for all $1\le i\le n$.

Two automorphisms $\zeta,\chi\in\Aut(A[D_n])$ play a central role in our study.
These are defined by
\begin{gather*}
\zeta(t_i)=t_i\text{ for }1\le i\le n-2\,,\ \zeta(t_{n-1})=t_n\,,\ \zeta(t_n)=t_{n-1}\,,\\
\chi(t_i)=t_i^{-1}\text{ for } 1\le i\le n\,.
\end{gather*}
Both are of order $2$ and commute, hence they generate a subgroup of $\Aut(A[D_n])$ isomorphic to $\Z/2\Z\times\Z/2\Z$.
If $n$ is odd, then $\zeta$ is the conjugation map by $\Delta=\Delta[D_n]$.
On the other hand, if $n$ is even, then $\zeta$ is not an inner automorphism (see Paris \cite{Paris1}).
The automorphism $\chi$ is never inner.

Two other homomorphisms play an important role in our study.
The first, $\pi:A[D_n]\to A[A_{n-1}]$, is defined by
\[
\pi(t_i)=s_i\text{ for }1\le i\le n-2\,,\ \pi(t_{n-1})=\pi(t_n)=s_{n-1}\,.
\]
The second, $\iota:A[A_{n-1}]\to A[D_n]$, is defined by
\[
\iota(s_i)=t_i\text{ for }1\le i\le n-1\,.
\]
Observe that $\pi\circ\iota=\id_{A[A_{n-1}]}$, hence $\pi$ is surjective, $\iota$ is injective, and $A[D_n]\simeq\Ker(\pi)\rtimes A[A_{n-1}]$.
We refer to Crisp--Paris \cite{CriPar1} for a detailed study on this decomposition of $A[D_n]$ as a semi-direct product.

Let $n\ge 4$.
For $p\in\Z$ we define a homomorphism $\alpha_p:A[D_n]\to A[A_{n-1}]$ by
\[
\alpha_p(t_i)=s_i\Delta^{2p}\text{ for }1\le i\le n-2\,,\ \alpha_p(t_{n-1})=\alpha_p(t_n)=s_{n-1}
\Delta^{2p}\,,
\]
where $\Delta=\Delta[A_{n-1}]$ is the Garside element of $A[A_{n-1}]$.
Note that $\alpha_0=\pi$.

Set $Y=\{t_1,\dots,t_{n-1}\}$.
For $p,q\in\Z$ we define a homomorphism $\beta_{p,q}:A[A_{n-1}]\to A[D_n]$ by
\[
\beta_{p,q}(s_i)=t_i\Delta_Y^{2p}\Delta^{\kappa q}\text{ for }1\le i\le n-1\,,
\]
where $\Delta=\Delta[D_n]$ is the Garside element of $A[D_n]$, $\Delta_Y=\Delta_Y[D_n]$, $\kappa=2$ if $n$ is odd, and $\kappa=1$ if $n$ is even.
Note that $\beta_{0,0}=\iota$.
Note also that, by Paris \cite[Theorem 1.1]{Paris1}, the centralizer of $Y$ in $A[D_n]$ is the free abelian group of rank $2$ generated by $\Delta_Y^2$ and $\Delta^\kappa$.

For $p\in\Z$ we define the homomorphism $\gamma_p:A[D_n]\to A[D_n]$ by
\[
\gamma_p(t_i)=t_i\Delta^{\kappa p}\text{ for }1\le i\le n\,,
\]
where $\Delta=\Delta[D_n]$ is the Garside element of $A[D_n]$, $\kappa=2$ if $n$ is odd, and $\kappa=1$ if $n$ is even.
Note that $\gamma_0=\id$.

Concerning $A [A_{n-1}]$, we define an automorphism $\bar\chi:A[A_{n-1}]\to A[A_{n-1}]$ by
\[
\bar\chi(s_i)=s_i^{-1}\text{ for }1\le i\le n-1\,,
\]
and for $p\in\Z$ we define an endomorphism $\bar\gamma_p:A[A_{n-1}]\to A[A_{n-1}]$ by
\[
\bar\gamma_p(s_i)=s_i\Delta^{2p}\text{ for }1\le i\le n-1\,,
\]
where $\Delta$ is the Garside element of $A[A_{n-1}]$.

The main results of this paper are the following.

\begin{thm}\label{thm2_1}
Let $n\ge 5$.
Let $\varphi:A[D_n]\to A[A_{n-1}]$ be a homomorphism.
Then up to conjugation we have one of the following two possibilities.
\begin{itemize}
\item[(1)]
$\varphi$ is cyclic.
\item[(2)]
There exist $\psi\in\langle\chi\rangle$ and $p\in\Z$ such that $\varphi=\alpha_p\circ\psi$.
\end{itemize}
\end{thm}

\begin{thm}\label{thm2_2}
Let $n\ge6$.
Let $\varphi:A[A_{n-1}]\to A[D_n]$ be a homomorphism.
Then up to conjugation we have one of the following two possibilities.
\begin{itemize}
\item[(1)]
$\varphi$ is cyclic.
\item[(2)]
There exist $\psi\in\langle\zeta,\chi\rangle$ and $p,q\in\Z$ such that $\varphi=\psi\circ\beta_{p,q}$.
\end{itemize}
\end{thm}

\begin{thm}\label{thm2_3}
Let $n\ge6$.
Let $\varphi:A[D_n]\to A[D_n]$ be a homomorphism.
Then up to conjugation we have one of the following three possibilities.
\begin{itemize}
\item[(1)]
$\varphi$ is cyclic.
\item[(2)]
There exist $\psi\in\langle\zeta,\chi\rangle$ and $p,q\in\Z$ such that $\varphi=\psi\circ\beta_{p,q}\circ\pi$.
\item[(3)]
There exist $\psi\in\langle\zeta,\chi\rangle$ and $p\in\Z$ such that $\varphi=\psi\circ\gamma_p$.
\end{itemize}
\end{thm}

From Theorem \ref{thm2_3} we deduce a classification of the injective endomorphisms and of the automorphisms of $A[D_n]$ as follows.

\begin{corl}\label{corl2_4}
Let $n\ge6$.
Let $\varphi:A[D_n]\to A[D_n]$ be an endomorphism.
Then $\varphi$ is injective if and only if there exist $\psi\in\langle\zeta,\chi\rangle$ and $p\in\Z$ such that $\varphi$ is conjugate to $\psi\circ\gamma_p$.
\end{corl}

\begin{proof}
Let $\varphi:A[D_n]\to A[D_n]$ be an endomorphism.
By Theorem \ref{thm2_3} we have one of the following three possibilities up to conjugation.
\begin{itemize}
\item[(1)]
$\varphi$ is cyclic.
\item[(2)]
There exist $\psi\in\langle\zeta,\chi\rangle$ and $p,q\in\Z$ such that $\varphi=\psi\circ\beta_{p,q}\circ\pi$.
\item[(3)]
There exist $\psi\in\langle\zeta,\chi\rangle$ and $p\in\Z$ such that $\varphi=\psi\circ\gamma_p$.
\end{itemize}
If $\varphi$ is cyclic, then $\varphi(t_{n-1})=\varphi(t_n)$, hence $\varphi$ is not injective.
If there exist $\psi\in\langle\zeta,\chi\rangle$ and $p,q\in\Z$ such that $\varphi=\psi\circ\beta_{p,q}\circ\pi$, then, again, $\varphi(t_{n-1})=\varphi(t_n)$, hence $\varphi$ is not injective.
So, if $\varphi$ is injective, then there exist $\psi\in\langle\zeta,\chi\rangle$ and $p\in\Z$ such that $\varphi$ is conjugate to $\psi\circ\gamma_p$.

It remains to show that, if $\psi\in\langle\zeta,\chi\rangle$ and $p\in\Z$, then $\psi\circ\gamma_p$ is injective.
Since the elements of $\langle\zeta,\chi\rangle$ are automorphisms, it suffices to show that $\gamma_p$ is injective.
We denote by $z:A[D_n]\to\Z$ the homomorphism which sends $t_i$ to $1$ for all $1\le i\le n$.
It is easily seen that $\gamma_p(u)=u\Delta^{\kappa p\,z(u)}$ for all $u\in A[D_n]$.
Let $u\in\Ker(\gamma_p)$.
Then $1=\gamma_p(u)=u\Delta^{\kappa p\,z(u)}$, hence $u=\Delta^q$ where $q=-\kappa p\,z(u)$.
We have $z(\Delta)=n(n-1)$, hence $z(u)=qn(n-1)$, thus
\[
1=\gamma_p(u)=\Delta^q\Delta^{\kappa pqn(n-1)}=\Delta^{q(1+\kappa pn(n-1))}\,.
\]
Since $1+\kappa pn(n-1)\neq 0$, this equality implies that $q=0$, hence $u=1$.
So, $\gamma_p$ is injective.
\end{proof}

\begin{corl}\label{corl2_5}
Let $n\ge6$.
Let $\varphi:A[D_n]\to A[D_n]$ be an endomorphism.
Then $\varphi$ is an automorphism if and only if it is conjugate to an element of $\langle\zeta,\chi\rangle$.
\end{corl}

\begin{proof}
Clearly, if $\varphi$ is conjugate to an element of $\langle\zeta,\chi\rangle$, then $\varphi$ is an automorphism.
Conversely, suppose that $\varphi$ is an automorphism.
We know from Corollary \ref{corl2_4} that there exist $\psi\in\langle\zeta,\chi\rangle$ and $p\in\Z$ such that $\varphi$ is conjugate to $\psi\circ\gamma_p$.
Thus, up to conjugation and up to composing on the left by $\psi^{-1}$, we can assume that $\varphi=\gamma_p$.
It remains to show that $p=0$.

Let again $z:A[D_n]\to\Z$ be the homomorphism which sends $t_i$ to $1$ for all $1\le i\le n$.
Recall that $\gamma_p(u)=u\Delta^{\kappa p\,z(u)}$ for all $u\in A[D_n]$.
For $u\in A[D_n]$, we have
\[
(z\circ\gamma_p)(u)=(1+n(n-1)\kappa p)z(u)\in(1+n(n-1)\kappa p)\Z\,.
\]
Since $\gamma_p$ is an automorphism, $z\circ\gamma_p$ is surjective, hence $\Z=\Im(z\circ\gamma_p)\subset (1+n(n-1)\kappa p)\Z$.
It follows that $(1+n(n-1)\kappa p)\in\{\pm 1\}$, hence $p=0$.
\end{proof}

By combining Corollary \ref{corl2_5} with Crisp--Paris \cite[Theorem 4.9]{CriPar1} we immediately obtain the following.

\begin{corl}\label{corl2_6}
Let $n\ge 6$.
\begin{itemize}
\item[(1)]
If $n$ is even, then
\[
\Aut(A[D_n])=\Inn(A[D_n])\rtimes\langle\zeta,\chi\rangle\simeq(A[D_n]/Z(A[D_n]))\rtimes(\Z /2\Z\times\Z/2\Z)\,,
\]
and $\Out(A[D_n])\simeq\Z/2\Z\times\Z/2\Z$, where $Z(A[D_n])$ denotes the center of $A[D_n]$.
\item[(2)]
If $n$ is odd, then
\[
\Aut(A[D_n])=\Inn(A[D_n])\rtimes\langle\chi\rangle\simeq(A[D_n]/Z(A[D_n]))\rtimes(\Z/2\Z)\,,
\]
and $\Out(A[D_n])\simeq\Z/2\Z$.
\end{itemize}
\end{corl}

We denote by $Z (A [D_n])$ the center of $A [D_n]$, we set $A_Z [D_n] = A [D_n] / Z (A [D_n])$, and we denote by $\xi : A [D_n] \to A_Z [D_n]$ the canonical projection. 
For each $1 \le i \le n$, we set $t_{Z,i} = \xi (t_i)$.
Note that an endomorphism $\varphi : A [D_n] \to A [D_n]$ induces an endomorphism $\varphi_Z : A_Z [D_n] \to A_Z [D_n]$ if and only if $\varphi (Z (A [D_n])) \subset Z (A[D_n])$. 
We say that an endomorphism $\psi: A_Z [D_n] \to A_Z [D_n]$ \emph{lifts} if there exists an endomorphism
$\varphi : A [D_n] \to A [D_n]$ such that $\varphi_Z = \psi$.
Then we call $\varphi$ a \emph{lift} of $\psi$.
In Section \ref{sec7} we prove the following.

\begin{prop}\label{propA1}
Let $n \ge 4$. 
Then every endomorphism of $A_Z [D_n]$ lifts.
\end{prop}

From this proposition combined with Theorem \ref{thm2_3} we will deduce the following.

\begin{thm}\label{thmA2}
Let $n \ge 6$. 
Let $\varphi_Z : A_Z [D_n] \to A_Z [D_n]$ be an endomorphism. 
Then we have one of the following two possibilities up to conjugation.
\begin{itemize}
\item[(1)]
$\varphi_Z$ is cyclic.
\item[(2)]
$\varphi_Z \in \langle \zeta_Z, \chi_Z \rangle$.
\end{itemize}
\end{thm}

In addition to Theorem \ref{thmA2} we have the following.

\begin{prop}\label{propA7}
Let $n \ge 4$. 
There are only finitely many conjugacy classes of cyclic endomorphisms of $A_Z [D_n]$.
\end{prop}

\begin{proof}
Let $\varphi_Z : A_Z [D_n] \to A_Z [D_n]$ be a cyclic endomorphism. 
There exists $g_Z \in A_Z [D_n]$ such that $\varphi_Z (t_{Z,i}) = g_Z$ for all $1 \le i \le n$. 
We denote by $\Delta$ the Garside element of $A [D_n]$, and we set $\kappa = 2$ if $n$ is odd and $\kappa = 1$ if $n$ is even. 
We have $1 = (\varphi_Z \circ \xi) (\Delta^\kappa) = g_Z^{\kappa n (n-1)}$, hence $g_Z$ is of finite order.
By Bestvina \cite[Theorem 4.5]{Bestv1} there are finitely many conjugacy classes of finite subgroups in $A_Z[D_n]$. 
Since $\langle g_Z \rangle$ is a finite subgroup of $A_Z [D_n]$, it follows that there are finitely many choices for $g_Z$ up to conjugation. 
\end{proof}

In Lemma \ref{lemA6} we will show that, if $n$ is even then $\langle \zeta_Z, \chi_Z \rangle \cap \Inn (A_Z [D_n]) = \{ \id\}$, and if $n$ is odd then $\langle \chi_Z \rangle \cap \Inn (A_Z [D_n]) = \{ \id \}$.
Furthermore, it is well-known that the center of $A_Z [D_n]$ is trivial (see Cumplido--Paris \cite[Proposition 6]{CumPar1} for example). 
These two remarks combined with Theorem \ref{thmA2} imply the following.

\begin{corl}\label{corlA3}
Let $n \ge 6$.
\begin{itemize}
\item[(1)]
If $n$ is even, then
\[
\Aut (A_Z [D_n]) = \Inn (A_Z [D_n]) \rtimes \langle \zeta_Z, \chi_Z \rangle \simeq A_Z [D_n] \rtimes (\Z / 2 \Z \times \Z / 2 \Z) \simeq \Aut (A [D_n])\,,
\]
and $\Out (A_Z [D_n]) \simeq \Z / 2 \Z \times \Z / 2\Z \simeq \Out (A [D_n])$.
\item[(2)]
If $n$ is odd, then 
\[
\Aut (A_Z [D_n]) = \Inn (A_Z [D_n]) \rtimes \langle \chi_Z \rangle \simeq A_Z [D_n] \rtimes (\Z / 2 \Z) \simeq \Aut (A [D_n])\,,
\]
and $\Out (A_Z [D_n]) \simeq \Z / 2 \Z \simeq \Out (A [D_n])$.
\end{itemize}
\end{corl}

A group $G$ is said to be \emph{co-Hopfian} if every injective endomorphism of $G$ is an isomorphism. 
Another direct consequence of Theorem \ref{thmA2} is the following.

\begin{corl}\label{corlA4}
Let $n \ge 6$. 
Then $A_Z[D_n]$ is co-Hopfian.
\end{corl}

Note that, in addition to the case $D_n$ for $n \ge 6$ shown in Corollary \ref{corlA4}, the only Coxeter graphs $\Gamma$ for which we know that $A[\Gamma]/ Z (A[\Gamma])$ is co-Hopfian are the Coxeter graphs $\Gamma = A_n$ for $n \ge 2$ (see Bell--Margalit \cite{BelMar1}).


\section{Geometric representations}\label{sec3}

Let $\Sigma$ be an oriented compact surface possibly with boundary, and let $\PP$ be a finite set of punctures  in the interior of $\Sigma$.
We denote by $\Homeo^+(\Sigma,\PP)$ the group of homeomorphisms of $\Sigma$ that preserve the orientation, that are the identity on a neighborhood of the boundary of $\Sigma$, and that setwise leave invariant $\PP$.
The \emph{mapping class group} of the pair $(\Sigma,\PP)$, denoted $\MM(\Sigma,\PP)$, is the group of isotopy classes of elements of $\Homeo^+(\Sigma,\PP)$.
If $\PP=\emptyset$, then we write $\MM(\Sigma,\emptyset)=\MM(\Sigma)$, and if $\PP=\{x\}$ is a singleton, then we write $\MM(\Sigma,\PP)=\MM(\Sigma,x)$.
We only give definitions and results on mapping class groups that we need for our proofs and we refer to Farb--Margalit \cite{FarMar1} for a complete account on the subject.

Recall that a \emph{geometric representation} of an Artin group $A$ is a homomorphism from $A$ to a mapping class group.
Their study is the main ingredient of our proofs. 
Important tools for constructing and understanding them are Dehn twists and essential reduction systems.
So, we start by recalling their definitions and their main properties.

A \emph{circle} of $(\Sigma,\PP)$ is the (non-oriented) image of an embedding $a:\S^1\hookrightarrow\Sigma\setminus(\partial\Sigma\cup\PP)$.
It is called \emph{generic} if it does not bound any disk containing $0$ or $1$ puncture and if it is not parallel to any boundary component.
The isotopy class of a circle $a$ is denoted by $[a]$.
We denote by $\CC(\Sigma,\PP)$ the set of isotopy classes of generic circles of $(\Sigma,\PP)$.
The \emph{intersection number} of two classes $[a],[b]\in\CC(\Sigma,\PP)$ is $i([a],[b])=\min\{|a'\cap b'|\mid a'\in[a]\text{ and }b'\in[b]\}$.
The set $\CC(\Sigma,\PP)$ is endowed with a simplicial complex structure, where a finite set $\AA$ is a simplex if $i([a],[b])=0$ for all $[a],[b]\in\AA$.
This complex is called the \emph{curve complex} of $(\Sigma,\PP)$.

In this paper by a \emph{Dehn twist} we mean a right Dehn twist and the (right) Dehn twist along a circle $a$ of $(\Sigma,\PP)$ will be denoted by $T_a$.
The following is an important tool for constructing and understanding geometric representations of Artin groups.
Its proof can be found in Farb--Margalit \cite[Section 3.5]{FarMar1}.

\begin{prop}\label{prop3_0}
Let $\Sigma$ be a compact oriented surface and let $\PP$ be a finite collection of punctures in the interior of $\Sigma$.
Let $a,b$ be two generic circles of $(\Sigma,\PP)$.
\begin{itemize}
\item[(1)]
We have $T_aT_b=T_bT_a$ if and only if $i([a],[b])=0$.
\item[(2)]
We have $T_aT_bT_a=T_bT_aT_b$ if and only if $i([a],[b])=1$.
\end{itemize}
\end{prop}

Let $f\in\MM(\Sigma,\PP)$.
A simplex $\AA$ of $\CC(\Sigma,\PP)$ is called a \emph{reduction system} for $f$ if $f(\AA)=\AA$.
In that case any element of $\AA$ is called a \emph{reduction class} for $f$.
A reduction class $[a]$ is an \emph{essential reduction class} if, for all $[b]\in\CC(\Sigma,\PP)$ such that $i([a],[b])\neq 0$ and for all $m\in\Z\setminus\{0\}$, we have $f^m([b])\neq[b]$.
In particular, if $[a]$ is an essential reduction class and $[b]$ is any reduction class, then $i([a],[b])=0$.
We denote by $\SS(f)$ the set of essential reduction classes for $f$.
The following gathers some key results on $\SS(f)$ that will be useful later.

\begin{thm}[Birman--Lubotzky--McCarthy \cite{BiLuMc1}]\label{thm3_1}
Let $\Sigma$ be a compact oriented surface and let $\PP$ be a finite set of punctures in the interior of $\Sigma$.
Let $f\in\MM(\Sigma,\PP)$.
\begin{itemize}
\item[(1)]
If $\SS(f)\neq\emptyset$, then $\SS(f)$ is a reduction system for $f$.
In particular, if $\SS(f)\neq\emptyset$, then $\SS(f)$ is a simplex of $\CC(\Sigma,\PP)$.
\item[(2)]
We have $\SS(f^n)=\SS(f)$ for all $n\in\Z\setminus\{0\}$.
\item[(3)]
We have $\SS(gfg^{-1}) = g(\SS(f))$ for all $g\in\MM(\Sigma,\PP)$.
\end{itemize}
\end{thm}

The following is well-known and it is a direct consequence of Birman--Lubotzky--McCarthy \cite{BiLuMc1} (see also Castel \cite[Corollaire 3.45]{Caste2}).
It will be often used in our proofs. 

\begin{prop}\label{prop3_2}
Let $\Sigma$ be an oriented compact surface of genus $\ge2$ and let $\PP$ be a finite set of punctures in the interior of $\Sigma$.
Let $f_0\in Z(\MM(\Sigma,\PP))$ be a central element of $\MM(\Sigma,\PP)$, let $\AA=\{[a_1],\dots,[a_p]\}$ be a simplex of $\CC(\Sigma,\PP)$, and let $k_1,\dots,k_p$ be nonzero integers.
Let $g=T_{a_1}^{k_1}T_{a_2}^{k_2}\cdots T_{a_p}^{k_p}f_0$.
Then $\SS(g)=\AA$.
\end{prop}

Let $n\ge 4$.
If $n$ is even, then $\Sigma_n$ denotes the surface of genus $\frac{n-2}{2}$ with two boundary components, and if $n$ is odd, then $\Sigma_n$ denotes the surface of genus $\frac{n-1}{2}$ with one boundary component.
Consider the circles $a_1,\dots,a_{n-1}$ drawn in Figure \ref{fig3_1}.
Then by Proposition \ref{prop3_0} we have a geometric representation $\rho_A:A[A_{n-1}]\to\MM(\Sigma_n)$ which sends $s_i$ to $T_{a_i}$ for all $1\le i\le n-1$.
The following is well-known, it is a direct consequence of Birman--Hilden \cite{BirHil1}, and its proof is explicitly given in Perron--Vannier \cite{PerVan1}.

\begin{figure}[ht!]
\begin{center}
\parbox[c]{7.1cm}{\includegraphics[width=7cm]{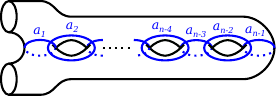}} \quad $n$ even

\bigskip
\parbox[c]{6.7cm}{\includegraphics[width=6.6cm]{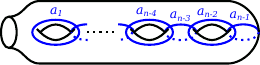}} \quad $n$ odd
\caption{Geometric representation of $A[A_{n-1}]$}\label{fig3_1}
\end{center}
\end{figure}

\begin{thm}[Birman--Hilden \cite{BirHil1}]\label{thm3_3}
Let $n\ge4$.
Then $\rho_A:A[A_{n-1}]\to\MM(\Sigma_n)$ is injective.
\end{thm}

The following is proved in Castel \cite{Caste2} for $n\ge6$ using the geometric representation $\rho_A$ defined above.
It is proved in Chen--Kordek--Margalit \cite{ChKoMa1} for $n\ge5$ with a different method.

\begin{thm}[Castel \cite{Caste2}, Chen--Kordek--Margalit \cite{ChKoMa1}, Orevkov \cite{Orevk1}]\label{thm3_4}
Let $n\ge5$.
Let $\varphi:A[A_{n-1}]\to A[A_{n-1}]$ be a homomorphism.
Then up to conjugation we have one of the following two possibilities.
\begin{itemize}
\item[(1)]
$\varphi$ is cyclic.
\item[(2)]
There exist $\psi\in\langle\bar\chi\rangle$ and $p\in\Z$ such that $\varphi=\psi\circ\bar\gamma_p$.
\end{itemize}
\end{thm}

Let $n\ge6$.
Pick a puncture $x$ in the interior of $\Sigma_n$ and consider the circles $d_1,\dots,d_n$ drawn in Figure \ref{fig3_2}.
Then by Proposition \ref{prop3_0} we have a geometric representation $\rho_D:A[D_n]\to\MM(\Sigma_n,x)$ which sends $t_i$ to $T_{d_i}$ for all $1\le i\le n$.
On the other hand, the embedding of $\Homeo^+(\Sigma_n,x)$ into $\Homeo^+(\Sigma_n)$ induces a surjective homomorphism $\theta:\MM(\Sigma_n,x)\to\MM(\Sigma_n)$ whose kernel is naturally isomorphic to $\pi_1(\Sigma_n,x)$ (see Birman \cite{Birma1}).
It is easily seen that
\[
\theta(T_{d_i})=T_{a_i}\text{ for }1\le i\le n-2\,,\ \theta(T_{d_{n-1}})=\theta(T_{d_n})=T_{a_{n-1}}\,,
\]
hence we have the following commutative diagram:
\begin{equation}\label{eq3_1}
\xymatrix{
1\ar[r]&\Ker(\pi)\ar[r]\ar[d]^{\bar \rho}&A[D_n]\ar[r]^\pi\ar[d]^{\rho_D}&A[A_{n-1}]\ar[r]\ar[d]^{\rho_A}&1\\
1\ar[r]&\Ker(\theta)\ar[r]&\MM(\Sigma_n,x)\ar[r]^\theta&\MM(\Sigma_n)\ar[r]&1}
\end{equation}
where we denote by $\bar\rho:\Ker(\pi)\to\Ker(\theta)$ the restriction of $\rho_D$ to $\Ker(\pi)$.

\begin{figure}[ht!]
\begin{center}
\parbox[c]{7.1cm}{\includegraphics[width=7cm]{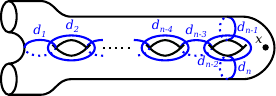}} \quad $n$ even

\bigskip
\parbox[c]{6.7cm}{\includegraphics[width=6.6cm]{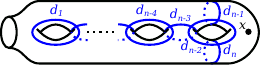}} \quad $n$ odd
\caption{Geometric representation of $A[D_n]$}\label{fig3_2}
\end{center}
\end{figure}

The proof of the following can be found in Perron--Vannier \cite[Theorem 1]{PerVan1} with few modifications.
As this result is central in our paper, for the sake of completeness we give a proof.
Note that our proof is a little shorter than that of Perron--Vannier \cite{PerVan1} because it uses results from Crisp--Paris \cite{CriPar1} which were not known and it does not need to deal with some Dehn twist along a boundary component, but our arguments are essentially the same.

\begin{thm}[Perron--Vannier \cite{PerVan1}]\label{thm3_5}
Let $n\ge4$.
\begin{itemize}
\item[(1)]
The homomorphism $\bar\rho:\Ker(\pi)\to\Ker(\theta)$ is an isomorphism.
\item[(2)]
The geometric representation $\rho_D:A[D_n]\to\MM(\Sigma_n,x)$ is injective.
\end{itemize}
\end{thm}

\begin{proof}
Part (2) is a consequence of Part (1) because of the following.
Suppose $\bar\rho$ is an isomorphism.
Then, since $\rho_A$ is injective, $\rho_D$ is injective by the five lemma applied to diagram (\ref{eq3_1}).

Now, we prove Part (1).
We know from Crisp--Paris \cite[Proposition 2.3]{CriPar1} that $\Ker(\pi)$ is a free group of rank $n-1$.
We also know from Birman \cite{Birma1} that $\Ker(\theta)=\pi_1(\Sigma_n,x)$, which is also a free group of rank $n-1$.
Recall that a group $G$ is Hopfian if every surjective endomorphism $G\to G$ is an isomorphism.
It is well-known that free groups of finite rank are Hopfian (see de la Harpe \cite[Chapter III, Section 19]{Harpe1}), hence in order to show that $\bar\rho$ is an isomorphism it suffices to show that $\bar\rho$ is surjective.

Set $f_{n-1}=T_{d_{n-1}}^{-1}T_{d_n}$.
Note that $t_{n-1}^{-1}t_n\in\Ker(\pi)$ and $f_{n-1}=\bar\rho(t_{n-1}^{-1}t_n)$.
In particular $f_{n-1}\in\Im(\bar\rho)\subset\Ker(\theta)=\pi_1(\Sigma_n,x)$.
This element, seen as an element of $\pi_1(\Sigma_n,x)$, is represented by the loop drawn in Figure \ref{fig3_3}.
For $2\le i\le n-1$ we define $f_{n-i}\in\pi_1(\Sigma_n,x)\subset\MM(\Sigma_n,x)$ by induction on $i$ by setting $f_{n-i}=T_{d_{n-i}}\,f_{n-i+1}T_{d_{n-i}}^{-1}\,f_{n-i+1}^{-1}$.
The element $f_{n-i}$, viewed as an element of $\pi_1(\Sigma_n,x)$, is represented by the loop drawn in the left-hand side of Figure \ref{fig3_4} if $i=2j$ is even, and by the loop drawn in the right-hand side of Figure \ref{fig3_4} if $i=2j+1$ is odd, where we compose paths from right to left.
Observe that $f_1,\dots,f_{n-1}$ generate $\pi_1(\Sigma_n,x)$.
So, in order to show that $\bar\rho$ is surjective, it suffices to show that $f_{n-i}\in\Im(\bar\rho)$ for all $i\in\{1,\dots,n-1\}$.
We argue by induction on $i$.
We already know that $f_{n-1}=\bar\rho(t_{n-1}^{-1}t_n)\in\Im(\bar\rho)$.
Suppose $i\ge 2$ and $f_{n-i+1}\in\Im(\bar\rho)$.
Let $u\in\Ker(\pi)$ such that $f_{n-i+1}=\bar\rho(u)$.
Since $\Ker(\pi)$ is a normal subgroup of $A[D_n]$, we have $t_{n-i}ut_{n-i}^{-1}\in\Ker(\pi)$, hence $t_{n-i}ut_{n-i}^{-1}u^{-1}\in\Ker(\pi)$, and therefore 
\[
f_{n-i}=T_{d_{n-i}}\,f_{n-i+1}T_{d_{n-i}}^{-1}\,f_{n-i+1}^{-1}=\bar\rho(t_{n-i}ut_{n-i}^{-1}u^{-1})
\in\Im(\bar\rho)\,.\proved
\]
\end{proof}

\begin{figure}[ht!]
\begin{center}
\parbox[c]{8.1cm}{\includegraphics[width=8cm]{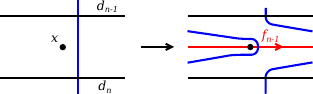}}
\hskip2cm
\parbox[c]{2.5cm}{\includegraphics[width=2.4cm]{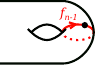}}
\caption{The loop $f_{n-1}\in\pi_1(\Sigma_n,x)$}\label{fig3_3}
\end{center}
\end{figure}

\begin{figure}[ht!]
\begin{center}
\parbox[c]{4.9cm}{\includegraphics[width=4.8cm]{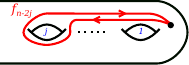}}
\hskip2cm
\parbox[c]{5.3cm}{\includegraphics[width=5.2cm]{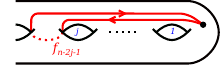}}
\caption{The loop $f_{n-i}\in\pi_1(\Sigma_n,x)$}\label{fig3_4}
\end{center}
\end{figure}

Our last preliminary on geometric representations is a result implicitly proved in Castel \cite[Section 3.2]{Caste1}, and it is in this theorem that we need the assumption $n \ge 6$.

\begin{thm}[Castel \cite{Caste1}]\label{thm3_6}
Let $n\ge6$.
Let $\varphi:A[A_{n-1}]\to\MM(\Sigma_n,x)$ be a non-cyclic homomorphism.
Then there exist generic circles $c_1,\dots,c_{n-1}$ in $\Sigma_n\setminus\{x\}$, $\varepsilon\in\{\pm 1\}$ and $g\in\MM(\Sigma_n,x)$ such that
\begin{itemize}
\item[(a)]
$|c_i\cap c_j|=1$ if $|i-j|=1$ and $|c_i\cap c_j|=0$ if $|i-j|\ge2$, for all $1\le i,j\le n-1$,
\item[(b)]
$g$ commutes with $T_{c_i}$ for all $1\le i\le n-1$,
\item[(c)]
$\varphi(s_i)=T_{c_i}^{\varepsilon}g$ for all $1\le i\le n-1$.
\end{itemize}
\end{thm}

\begin{proof}
Assume $n$ is even. 
Let $\partial_1, \partial_2$ be the two boundary components of $\Sigma_n$. 
We denote by $\hat \Sigma_n$ the closed surface obtained from $\Sigma_n$ by gluing a disk $D_1$ along $\partial_1$ and a disk $D_2$ along $\partial_2$. 
Moreover, we choose a point $\hat x_1$ in the interior of $D_1$ and a point $\hat x_2$ in the interior of $D_2$ and we set $\hat \PP = \{x, \hat x_1, \hat x_2\}$. 
Assume $n$ is odd. 
Let $\partial$ be the boundary component of $\Sigma_n$. 
We denote by $\hat \Sigma_n$ the closed surface obtained from $\Sigma_n$ by gluing a disk $D$ along $\partial$.
Moreover, we choose a point $\hat x$ in the interior of $D$ and we set $\hat \PP = \{ x, \hat x\}$. 
For each $n$ we denote by $\PP \MM (\hat \Sigma_n, \hat \PP)$ the subgroup of $\MM (\hat \Sigma_n, \hat \PP)$ formed by the isotopy classes of elements in $\Homeo^+ (\hat \Sigma_n, \hat \PP)$ which pointwise fix $\hat \PP$. 
The embedding of $\Sigma_n$ into $\hat \Sigma_n$ induces a surjective homomorphism $\varpi : \MM (\Sigma_n, x) \to \PP \MM (\hat \Sigma_n, \hat \PP)$. 
If $n$ is even, then the kernel of $\varpi$ is the free abelian group of rank $2$ generated by $T_{\partial_1}$ and $T_{\partial_2}$, and if $n$ is odd, then the kernel of $\varpi$ is the cyclic group generated by $T_\partial$. 
In both cases $\Ker (\varpi)$ is contained in the center of $\MM (\Sigma_n, x)$.

Let $\varphi : A [A_{n-1}] \to \MM (\Sigma_n, x)$ be a non-cyclic homomorphism. 
Assume that $\varpi \circ \varphi$ is cyclic. 
Then there exists $\hat g \in \PP \MM (\hat \Sigma_n, \hat \PP)$ such that $(\varpi \circ \varphi) (s_i) = \hat g$ for all $1 \le i \le n-1$.
Let $g \in \MM (\Sigma_n, x)$ be such that $\varpi (g) = \hat g$. 
For each $1 \le i \le n-1$ there exists $h_i \in \Ker (\varpi) \subset Z (\MM (\Sigma_n, x))$ such that $\varphi (s_i) = g h_i$. 
Let $1 \le i \le n-2$. 
Then
\[
g^3 h_i^2 h_{i+1} = \varphi (s_i s_{i+1} s_i) = \varphi (s_{i+1} s_i s_{i+1}) = g^3 h_i h_{i+1}^2\,,
\]
hence $h_i = h_{i+1}$. 
This shows that $\varphi (s_i) = g h_1$ for all $1 \le i \le n-1$, hence that $\varphi$ is cyclic: contradiction. 
So, $\varpi \circ \varphi$ is not cyclic.

To differentiate Dehn twists in $\MM (\Sigma_n, x)$ from those in $\PP \MM (\hat \Sigma_n, \hat \PP)$, for a circle $c$ in $\hat \Sigma_n \setminus \hat \PP$ we denote by $\hat T_c$ the Dehn twist in $\PP \MM (\hat \Sigma_n, \hat \PP)$ along $c$. 
By Castel \cite[Theorem 1]{Caste1} there exist generic circles $c_1, \dots, c_{n-1}$ in $\hat \Sigma_n \setminus \hat \PP$, $\varepsilon \in \{ \pm 1 \}$ and $\hat g \in \PP \MM (\hat \Sigma_n, \hat \PP)$ such that
\begin{itemize}
\item[(1)]
$|c_i \cap c_j| = 1$ if $|i-j|=1$ and $|c_i \cap c_j|=0$ if $|i-j| \ge 2$, for all $1 \le i, j \le n-1$,
\item[(2)]
$\hat g$ commutes with $\hat T_{c_i}$ for all $1 \le i \le n-1$,
\item[(3)]
$(\varpi \circ \varphi) (s_i) = \hat T_{c_i}^\varepsilon \hat g$ for all $1 \le i \le n-1$.
\end{itemize}
Clearly, we can choose each $c_i$ siting in the interior of $\Sigma_n$. 
Let $g \in \MM (\Sigma_n, x)$ be such that $\varpi (g) = \hat g$. 
It is easily shown with Castel \cite[Lemma 3.2.1]{Caste1} that $g$ and $T_{c_i}$ commute for all $1 \le i \le n-1$. 
Furthermore, for each $1 \le i \le n-1$, there exists $h_i \in \Ker (\varpi) \subset Z (\MM (\Sigma_n, x))$ such that $\varphi (s_i) = T_{c_i}^\varepsilon g h_i$. 
Let $1 \le i \le n-2$. 
Then
\[
T_{c_i}^\varepsilon T_{c_{i+1}}^\varepsilon T_{c_i}^\varepsilon g^3 h_i^2 h_{i+1} = \varphi (s_i s_{i+1} s_i) = \varphi (s_{i+1} s_i s_{i+1}) = T_{c_{i+1}}^\varepsilon T_{c_i}^\varepsilon T_{c_{i+1}}^\varepsilon g^3 h_i h_{i+1}^2 = T_{c_i}^\varepsilon T_{c_{i+1}}^\varepsilon T_{c_i}^\varepsilon g^3 h_i h_{i+1}^2\,,
\]
hence $h_{i+1} = h_i$. 
So, there exists $h \in \Ker(\varpi)$ such that $\varphi (s_i) = T_{c_i}^\varepsilon g h$ and $g h$ commutes with $T_{c_i}$ for all $1 \le i \le n-1$.
\end{proof}


\section{Homomorphisms from $A[D_n]$ to $A[A_{n-1}]$}\label{sec4}

\begin{proof}[Proof of Theorem \ref{thm2_1}]
Let $n\ge 5$.
Let $\varphi:A[D_n]\to A[A_{n-1}]$ be a homomorphism.
By precomposing $\varphi$ with $\iota : A [A_{n-1}] \to A [D_n]$, we obtain a homomorphism $\varphi \circ \iota : A[A_{n-1}] \to A [D_n] \to A [A_{n-1}]$, hence, by Theorem \ref{thm3_4}, one of the following two possibilities holds.
\begin{itemize}
\item
$\varphi\circ\iota$ is cyclic. 
\item
There exist $\psi\in\langle\bar\chi\rangle$ and $p\in\Z$ such that $\varphi\circ\iota$ is conjugate to $\psi\circ\bar\gamma_p$.
\end{itemize}

Suppose $\varphi\circ\iota$ is cyclic.
Then there exists $u\in A[A_{n-1}]$ such that $(\varphi\circ\iota)(s_i)=\varphi(t_i)=u$ for all $1\le i\le n-1$.
Moreover,
\[
\varphi(t_n)=\varphi(t_{n-2}t_n)\,\varphi(t_{n-2})\,\varphi(t_n^{-1}t_{n-2}^{-1})=\varphi(t_{n-2}t_n)\,
\varphi(t_1)\,\varphi(t_n^{-1}t_{n-2}^{-1})=\varphi(t_1)=u\,,
\]
hence $\varphi$ is cyclic.

So, up to conjugating and replacing $\varphi$ by $\varphi\circ\chi$ if necessary, we can assume that there exists $p\in\Z$ such that $\varphi\circ\iota=\bar\gamma_p$.
This means that $\varphi(t_i)=(\varphi\circ\iota)(s_i)=s_i\Delta^{2p}$ for all $1\le i\le n-1$, where $\Delta$ is the Garside element of $A[A_{n-1}]$.
Now we turn to show that $\varphi=\alpha_p$.

Set $Y=\{s_1,\dots,s_{n-3}\}$.
By Paris \cite[Theorem 5.1]{Paris2} the centralizer of $\langle s_1,\dots,s_{n-3},s_{n-1}\rangle$ in $A[A_{n-1}]$ is generated by $\Delta^2$, $\Delta_Y^2$ and $s_{n-1}$, where $\Delta_Y=\Delta_Y[A_{n-1}]$.
These three elements pairwise commute and generate a copy of $\Z^3$.
Set $u=\varphi(t_n)$.
Since $u$ commutes with $\varphi(t_i)=s_i\Delta^{2p}$ for all $i\in\{1,\dots,n-3,n-1\}$ and $\Delta^2$ in central in $A[A_{n-1}]$, $u$ belongs to the centralizer of $\langle s_1,\dots,s_{n-3},s_{n-1}\rangle$, hence there exist $k_1,k_2,k_3\in\Z$ such that $u=s_{n-1}^{k_1}\Delta_Y^{2k_2}\Delta^{2k_3}$.

It is well-known that $A[A_{n-1}]$ is naturally isomorphic to the mapping class group $\MM(\D,\PP)$, where $\D$ denotes the disk and $\PP=\{x_1,\dots,x_n\}$ is a set of $n$ punctures in the interior of $\D$.
In this identification $s_{n-1}^2$ corresponds to the Dehn twist along the circle $c_1$ depicted in Figure \ref{fig4_1}, $\Delta_Y^2$ corresponds to the Dehn twist along the circle $c_2$ depicted in the same figure, and $\Delta^2$ corresponds to the Dehn twist along a circle parallel to $\partial\D$.
By Proposition \ref{prop3_2} we have $\SS(u^2)\subseteq\{c_1,c_2\}$, where $c_1\in\SS(u^2)$ if and only if $k_1\neq0$ and $c_2\in\SS(u^2)$ if and only if $k_2\neq0$.
We know that $\varphi(t_1^2)=s_1^2\Delta^{4p}$, hence $\SS(\varphi(t_1^2))$ is formed by a single circle containing two marked points in its interior.
Since $t_1^2$ and $t_n^2$ are conjugate, $\varphi(t_1^2)$ and $\varphi(t_n^2)=u^2$ are conjugate, hence, by Theorem \ref{thm3_1}, $\SS(u^2)$ is also formed by a single circle containing two marked points in its interior.
It follows that $\SS(u^2)=\{c_1\}$, hence $k_1\neq0$ and $k_2=0$.
It remains to show that $k_1=1$ and $k_3=p$.

\begin{figure}[ht!]
\begin{center}
\includegraphics[width=5.2cm]{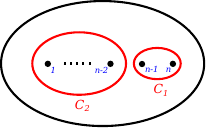}
\caption{Circles in the punctured disk}\label{fig4_1}
\end{center}
\end{figure}

From the equality $t_{n-2}t_nt_{n-2}=t_nt_{n-2}t_n$ it follows that $s_{n-2}s_{n-1}^{k_1}s_{n-2}\Delta^{4p+2k_3}=s_{n-1}^{k_1}s_{n-2}s_{n-1}^{k_1}\Delta^{2p+4k_3}$, hence $(s_{n-2}s_{n-1}^{k_1}s_{n-2})(s_{n-1}^{k_1}s_{n-2}s_{n-1}^{k_1})^{-1}=\Delta^{2k_3-2p}$.
We know from Paris \cite[Corollary 2.6]{Paris3} that $A_{\{s_{n-2},s_{n-1}\}}[A_{n-1}]\cap\langle\Delta\rangle=\{1\}$, hence $(s_{n-2}s_{n-1}^{k_1}s_{n-2})(s_{n-1}^{k_1}s_{n-2}s_{n-1}^{k_1})^{-1}=\Delta^{2k_3-2p}=1$.
Let $z:A[A_{n-1}]\to\Z$ be the homomorphism which sends $s_i$ to $1$ for all $1\le i\le n-1$.
We have
\[
0=z(1)=z((s_{n-2}s_{n-1}^{k_1}s_{n-2})(s_{n-1}^{k_1}s_{n-2}s_{n-1}^{k_1})^{-1})=1-k_1\,,
\]
hence $k_1=1$.
Moreover, $\Delta^{2k_3-2p}=1$ and $\Delta$ is of infinite order, thus $k_3=p$.
\end{proof}


\section{Homomorphisms from $A[A_{n-1}]$ to $A[D_n]$}\label{sec5}

The formula in the following lemma is a crucial point in various proofs, including that of Lemma \ref{lem5_3} and that of Theorem \ref{thmA2}.

\begin{lem}\label{lemA5}
Let $n \ge 1$. 
Then
\[
\Delta [A_n]^2 = (s_1 \cdots s_{n-1} s_n^2 s_{n-1} \cdots s_1) (s_2 \cdots s_{n-1} s_n^2 s_{n-1} \cdots s_2) \cdots (s_{n-1} s_n^2 s_{n-1}) s_n^2\,.
\]
\end{lem}  

\begin{proof}
We argue by induction on $n$. 
The case $n = 1$ is trivial, hence we can assume that $n \ge 2$ and that the induction hypothesis holds. 
Recall that 
\[
\Delta [A_n] = (s_1 \cdots s_n)\, \Delta [A_{n-1}] = \Delta [A_{n-1}]\, (s_n \cdots s_1)\,.
\]
Moreover, it is easily checked that $s_i (s_n \cdots s_1) = (s_n \cdots s_1) s_{i+1}$ for all $1 \le i \le n-1$.
By the induction hypothesis, 
\[
\Delta [A_{n-1}]^2 = (s_1 \cdots s_{n-2} s_{n-1}^2 s_{n-2} \cdots s_1) \cdots (s_{n-2} s_{n-1}^2 s_{n-2}) s_{n-1}^2\,.
\]
Hence
\begin{gather*}
\Delta [A_n]^2 =
(s_1 \cdots s_n) \Delta [A_{n-1}]^2 (s_n \cdots s_1) =\\
(s_1 \cdots s_n) \big( (s_1 \cdots s_{n-2} s_{n-1}^2 s_{n-2} \cdots s_1) \cdots (s_{n-2} s_{n-1}^2 s_{n-2}) s_{n-1}^2 \big) (s_n \cdots s_1) =\\
(s_1 \cdots s_n) (s_n \cdots s_1) \big( (s_2 \cdots s_{n-1} s_n^2 s_{n-1} \cdots s_2) \cdots (s_{n-1} s_n^2 s_{n-1}) s_n^2 \big) = \\
(s_1 \cdots s_{n-1} s_n^2 s_{n-1} \cdots s_1) \cdots (s_{n-1} s_n^2 s_{n-1}) s_n^2\,.
\end{gather*}
\end{proof}

Now, Lemmas \ref{lem5_1} to \ref{lem5_7} are preliminaries to the proof of Theorem \ref{thm2_2}.

\begin{lem}\label{lem5_1}
Let $n\ge6$.
Let $\varphi:A[A_{n-1}]\to A[D_n]$ be a homomorphism.
If $\pi\circ\varphi:A[A_{n-1}]\to A[A_{n-1}]$ is cyclic, then $\varphi$ is cyclic.
\end{lem}

\begin{proof}
Assume $\pi\circ\varphi$ is cyclic.
Then there exists $u\in A[A_{n-1}]$ such that $(\pi\circ\varphi)(s_i)=u$ for all $1\le i\le n-1$.
For $3\le i\le n-1$ we set $v_i=\varphi(s_is_1^{-1})$.
We have $\pi(v_i)=uu^{-1}=1$, hence $v_i\in\Ker(\pi)$.
We have
\[
(s_3s_1^{-1})(s_4s_1^{-1})(s_3s_1^{-1})=s_3s_4s_3s_1^{-3}=s_4s_3s_4s_1^{-3}=(s_4s_1^{-1})(s_3s_1^{-1})
(s_4s_1^{-1})\,,
\]
hence $v_3v_4v_3=v_4v_3v_4$.
Since $\Ker(\pi)$ is a free group (see Crisp--Paris \cite[Proposition 2.3]{CriPar1}) and two elements in a free group either freely generate a free group or commute, the existence of such equality implies that $v_3v_4=v_4v_3$.
It follows that $v_3v_4v_3=v_3v_4^2$, hence $v_3=v_4$, and therefore
\[
\varphi(s_3)\,\varphi(s_1)^{-1}=v_3=v_4=\varphi(s_4)\,\varphi(s_1)^{-1}\,.
\]
So, $\varphi(s_3)=\varphi(s_4)$.
We conclude by Castel \cite[Lemma 3.1.1]{Caste1} that $\varphi$ is cyclic.
\end{proof}

Let $n\ge6$.
If $n$ is odd, then $\Sigma_n$ has one boundary component that we denote by $\partial$, and we denote by $T_\partial$ the Dehn twist along $\partial$. 
If $n$ is even, then $\Sigma_n$ has two boundary components that we denote by $\partial_1$ and $\partial_2$, and we denote by $T_{\partial_1}$ and $T_{\partial_2}$ the Dehn twists along $\partial_1$ and $\partial_2$, respectively.
It is known that the center of $\MM(\Sigma_n)$, denoted by $Z(\MM(\Sigma_n))$, is the cyclic group generated by $T_{\partial}$ if $n$ is odd, and it is a free abelian group of rank two generated by $T_{\partial_1}$ and $T_{\partial_2}$ if $n$ is even (see Paris--Rolfsen \cite[Theorem 5.6]{ParRol1} for example).

\begin{lem}\label{lem5_2}
Let $n\ge2$.
Let $f\in\MM(\Sigma_n)$ such that $fT_{a_i}^2=T_{a_i}^2f$ for all $1\le i\le n-1$.
Then $f^2\in Z(\MM(\Sigma_n))$.
\end{lem}

\begin{proof}
Assume $n$ is odd.
The case where $n$ is even can be proved in the same way.
Let $f\in\MM(\Sigma_n)$ such that $fT_{a_i}^2=T_{a_i}^2f$ for all $1\le i\le n-1$.
Since $fT_{a_i}^2f^{-1}=T_{a_i}^2$ we have $f([a_i])=[a_i]$ (see Farb--Margalit \cite[Section 3.3]{FarMar1}).
The mapping class $f$ may reverse the orientation of each $a_i$ up to isotopy, but $f^2$ preserves the orientation of all $a_i$ up to isotopy, hence $f^2$ can be represented by an element of $\Homeo^+(\Sigma_n)$ which is the identity on a (closed) regular neighborhood $\Sigma'$ of $\bigcup_{i=1}^{n-1} a_i$.
We observe that $\Sigma'$ is a surface of genus $\frac{n-1}{2}$ with one boundary component, $\partial'$, and that $\partial\cup\partial'$ bounds a cylinder $C$.
This implies that $f^2\in\MM(C)\subset\MM(\Sigma_n)$.
Since $\MM(C)=\langle T_{\partial}\rangle=Z(\MM(\Sigma_n))$, we conclude that $f^2\in Z(\MM(\Sigma_n))$.
\end{proof}

\begin{lem}\label{lem5_3}
Let $n\ge 3$.
We set $m=n-1$ if $n$ is odd and $m=n-2$ if $n$ is even.
Let $1\le k\le m$.
Let $c$ be a generic circle of $\Sigma_n\setminus\{x\}$ such that $c\cap d_i=\emptyset$ for $1\le i \le k-2$, $|c\cap d_{k-1}|=1$ if $k\ge 2$, $c\cap d_k=\emptyset$, and $c$ is isotopic to $d_k$ in $\Sigma_n$.
Then there exists $g\in\Ker(\theta)$ such that $g([d_i])=[d_i]$ for all $1\le i\le k-1$ and $g([c])=[d_k]$.
\end{lem}

\begin{proof}
We identify $D_3$ with $A_3$ in this proof to treat the cases $k=2$ and $k=1$.
We first assume that $k$ is even.
If $c$ is isotopic in $\Sigma_n\setminus\{x\}$ to $d_k$, then it suffices to take $g=\id$.
So, we can assume that $c$ and $d_k$ are not isotopic in $\Sigma_n\setminus\{x\}$.
Since $c$ and $d_k$ are isotopic in $\Sigma_n$, by Epstein \cite[Lemma 2.4]{Epste1} there exists a cylinder $C$ in $\Sigma_n$ whose boundary components are $d_k$ and $c$.
Since $c$ and $d_k$ are not isotopic in $\Sigma_n\setminus\{x\}$, this cylinder must contain the puncture $x$.

Let $\Sigma'$ be a regular neighborhood of $\big(\bigcup_{i=1}^{k-1}d_i\big)\cup C$.
The surface $\Sigma'$ contains the cylinder $C$ with boundaries $c$ and $d_k$, having the puncture $x$ in it, and $d_{k-1}$ intersects $c$ and $d_k$ once. 
Hence an arc of the curve $d_{k-1}$ connects a point on $c$ with a point on $d_k$ within the cylinder $C$, and it may wind around the cylinder in different ways (see Figure \ref{fig5A1}).
However, by applying suitable Dehn twists about $c$ and $d_k$, one can unwind this arc to the
simplest case, shown in Figure \ref{fig5_1}. 
Hence, up to homeomorphism of the surface $\Sigma_n$, we may assume that the circles $d_1, \dots, d_k, c$ are arranged as in Figure \ref{fig5_1}.

\begin{figure}[ht!]
\begin{center}
\includegraphics[width=4.2cm]{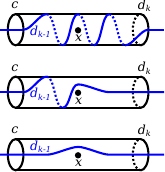}
\caption{Intersection of $C$ with $d_{k-1}$}\label{fig5A1}
\end{center}
\end{figure}

\begin{figure}[ht!]
\begin{center}
\includegraphics[width=6.6cm]{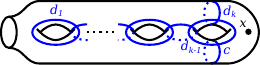}
\caption{Regular neighborhood of $\big(\bigcup_{i=1}^{k-1}d_i\big)\cup C$, case where $k$ is even}\label{fig5_1}
\end{center}
\end{figure}

By Proposition \ref{prop3_0} there are homomorphisms $\psi_1:A[D_{k+1}]\to\MM(\Sigma_n,x)$ and $\psi_2:A[A_k]\to\MM(\Sigma_n,x)$ defined by
\begin{gather*}
\psi_1(t_i)=T_{d_i}\text{ for }1\le i\le k\,,\ \psi_1(t_{k+1})=T_{c}\,,\\
\psi_2(s_i)=T_{d_i}\text{ for }1\le i\le k-1\,,\ \psi_2(s_k)=T_{c}\,.
\end{gather*}
We denote by $\Delta_{D,k}$ the Garside element of $A[D_{k+1}]$ and by $\Delta_{A,k}$ the Garside element of $A[A_k]$, and we set $g=\psi_1(\Delta_{D,k})\,\psi_2(\Delta_{A,k}^{-2})$.
We have $\Delta_{D,k}t_i\Delta_{D,k}^{-1}=t_i$ for all $1\le i\le k-1$, $\Delta_{D,k}t_{k+1}\Delta_{D,k}^{-1}=t_k$ and $\Delta_{A,k}^2s_i\Delta_{A,k}^{-2}=s_i$ for all $1\le i\le k$, hence $gT_{d_i}g^{-1}=T_{g(d_i)}=T_{d_i}$ for all $1\le i\le k-1$ and $gT_cg^{-1}=T_{g(c)}=T_{d_k}$.
It follows that $g([d_i])=[d_i]$ for all $1\le i\le k-1$ and $g([c])= [d_k]$ (see Farb--Margalit \cite[Fact 3.6]{FarMar1}).

Since $c$ and $d_k$ are isotopic in $\Sigma_n$, the corresponding Dehn twists $T_c$ and $T_{d_k}$ are
equal in $\MM (\Sigma_n)$, hence for $T_c$ and $T_{d_k}$, viewed on the surface $\Sigma_n \setminus \{ x \}$, we have $\theta (T_c) = \theta (T_{d_k})$.
Moreover,
\begin{gather*}
\Delta_{D,k}=(t_1\cdots t_{k-1}t_kt_{k+1}t_{k-1}\cdots t_1)\cdots(t_{k-1}t_kt_{k+1}t_{k-1})(t_kt_{k+1})\,,\\
\Delta_{A,k}^2=(s_1\cdots s_{k-1}s_k^2s_{k-1}\cdots s_1)\cdots(s_{k-1}s_k^2s_{k-1}) s_k^2\,,
\end{gather*}
(see Lemma \ref{lemA5} for the second equality), hence $\theta (\psi_1(\Delta_{D,k})) = \theta ( \psi_2 (\Delta_{A,k}^{2}))$, and therefore $\theta(g)=1$.
So, $g\in\Ker(\theta)$.

Now, assume $k$ is odd.
If $c$ is isotopic in $\Sigma_n\setminus\{x\}$ to $d_k$, then we can take $g=\id$.
So, we can assume that $c$ and $d_k$ are not isotopic in $\Sigma_n\setminus\{x\}$.
Since $c$ and $d_k$ are isotopic in $\Sigma_n$, there exists a cylinder $C$ in $\Sigma_n$ whose boundary components are $d_k$ and $c$.
Since $c$ and $d_k$ are not isotopic in $\Sigma_n\setminus\{x\}$, this cylinder must contain the puncture $x$.
Let $\Sigma'$ be a closed regular neighborhood of $\big(\bigcup_{i=1}^{k-1}d_i\big)\cup C$.
Then $\Sigma'$ is a surface of genus $\frac{k-1}{2}$ with two boundary components and the circles $d_1,\dots, d_{k-1},d_k,c$ are arranged as shown in Figure \ref{fig5_2}.
Since $k\le m$ and $k$ is odd, $\frac{k-1}{2}$ is strictly less than the genus of $\Sigma_n$, hence we can choose a sub-surface $\Sigma''$ of $\Sigma_n$ of genus $\frac{k+1}{2}$, with one boundary component, and containing $\Sigma'$.
We can also choose a generic circle $e$ in $\Sigma''\setminus\{x\}$ such that $|e\cap d_1|=1$, $|e\cap c|=1$ if $k=1$, $e\cap d_i=\emptyset$ for all $2\le i\le k$ and $e\cap c=\emptyset$ if $k\ge 2$ (see Figure \ref{fig5_2}).
By Proposition \ref{prop3_0} there are homomorphisms $\psi_1:A[D_{k+2}]\to\MM(\Sigma_n,x)$ and $\psi_2:A[A_{k+1}]\to\MM(\Sigma_n,x)$ defined by
\begin{gather*}
\psi_1(t_1)=T_e\,,\ \psi_1(t_i)=T_{d_{i-1}}\text{ for }2\le i\le k+1\,,\ \psi_1(t_{k+2})=T_{c}\,,\\
\psi_2(s_1)=T_e\,,\ \psi_2(s_i)=T_{d_{i-1}}\text{ for }2\le i\le k\,,\ \psi_2(s_{k+1})=T_{c}\,.
\end{gather*}
We denote by $\Delta_{D,k+1}$ the Garside element of $A[D_{k+2}]$ and by $\Delta_{A,k+1}$ the Garside element of $A[A_{k+1}]$, and we set $g=\psi_1(\Delta_{D,k+1})\,\psi_2(\Delta_{A,k+1}^{-2})$.
Then, as in the case where $k$ is even, we have $g([d_i])=[d_i]$ for all $1\le i \le k-1$, $g([c])=[d_k]$, and $g\in\Ker(\theta)$.
\end{proof}

\begin{figure}[ht!]
\begin{center}
\includegraphics[width=6.6cm]{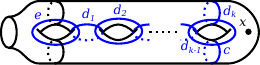}
\caption{Regular neighborhood of $\big(\bigcup_{i=1}^{k-1}d_i\big)\cup C$, case where $k$ is odd}\label{fig5_2}
\end{center}
\end{figure}

The following lemma is the extension of Lemma \ref{lem5_3} to the case $c \cap d_k \neq \emptyset$.

\begin{lem}\label{lem5_4}
Let $n\ge3$.
Set $m=n-1$ if $n$ is odd and $m=n-2$ if $n$ is even.
Let $1\le k\le m$.
Let $c$ be a generic circle of $\Sigma_n\setminus\{x\}$ such that $c\cap d_i=\emptyset$ for $1\le i\le k-2$, $|c\cap d_{k-1}|=1$ if $k\ge 2$, and $c$ is isotopic to $d_k$ in $\Sigma_n$.
Then there exists $g\in\Ker(\theta)$ such that $g([d_i])=[d_i]$ for all $1\le i\le k-1$ and $g([c])=[d_k]$.
\end{lem}

\begin{proof}
We argue by induction on $i([c],[d_k])$, which is computed on the surface $\Sigma_n \setminus \{ x\}$ and not on $\Sigma_n$.
The case $i([c],[d_k])=0$ is proved in Lemma \ref{lem5_3}, hence we can assume that $i([c],[d_k])\ge 1$ and that the induction hypothesis holds.
Note that now $c$ and $d_k$ cannot be isotopic in $\Sigma_n\setminus\{x\}$ since $i([c],[d_k])\neq 0$.
We can assume without loss of generality that $i([c],[d_k])=|c\cap d_k|$.
Since $c$ and $d_k$ are isotopic in $\Sigma_n$, there exists a bigon $D$ in $\Sigma_n$ cobounded by an arc of $d_k$ and an arc of $c$ as shown in Figure \ref{fig5_3}.
We can choose this bigon to be minimal in the sense that its interior intersects neither $c$ nor $d_k$.
The bigon $D$ cannot intersect $d_i$ for $1\le i\le k-2$ and one can easily modify $c$ so that $D$ does not intersect $d_{k-1}$ either.
Since $c$ and $d_k$ are not isotopic in $\Sigma_n\setminus\{x\}$, $D$ necessarily contains the puncture $x$ in its interior.
We choose a circle $c'$ parallel to $c$ except in the bigon $D$ where it follows the arc of $d_k$ which borders $D$ as illustrated in Figure \ref{fig5_3}.
By construction $c'\cap d_i=\emptyset$ for $1\le i\le k-2$, $|c'\cap d_{k-1}|=1$ if $k\ge 2$, and $c'$ is isotopic to $d_k$ in $\Sigma_n$.
Moreover $i([c'],[d_k])\le |c'\cap d_k|<|c\cap d_k|=i([c],[d_k])$.
By the induction hypothesis there exists $g_1\in\Ker(\theta)$ such that $g_1([d_i])=[d_i]$ for all $1\le i\le k-1$ and $g_1([c'])=[d_k]$.
By Farb--Margalit \cite[Lemma 2.9]{FarMar1}, we can choose $G_1\in\Homeo^+(\Sigma_n,x)$ which represents $g_1$ such that $G_1(d_i)=d_i$ for all $1\le i\le k-1$ and $G_1(c')=d_k$.
We set $c''=G_1(c)$.
Then $c''\cap d_i=\emptyset$ for $1\le i\le k-2$, $|c''\cap d_{k-1}|=1$ if $k\ge 2$, $c''\cap d_k=\emptyset$, and $c''$ is isotopic to $d_k$ in $\Sigma_n$.
By Lemma \ref{lem5_3} there exists $g_2\in\Ker(\theta)$ such that $g_2([d_i])=[d_i]$ for all $1\le i\le k-1$ and $g_2([c''])=[d_k]$.
We set $g=g_2\circ g_1$.
Then $g\in\Ker(\theta)$, $g([d_i])=[d_i]$ for all $1\le i\le k-1$ and $g([c])=[d_k]$.
\end{proof}

\begin{figure}[ht!]
\begin{center}
\includegraphics[width=3.6cm]{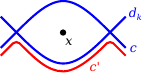}
\caption{Bigon cobounded by $c$ and $d_k$}\label{fig5_3}
\end{center}
\end{figure}

\begin{lem}\label{lem5_5}
Let $n$ be an even number, $n\ge 4$.
Let $c$ be a generic circle of $\Sigma_n\setminus\{x\}$ such that $c\cap d_i=\emptyset$ for all $1\le i\le n-3$, $|c\cap d_{n-2}|=1$, $c\cap d_{n-1}=\emptyset$, and $c$ is isotopic to $d_{n-1}$ in $\Sigma_n$.
Then we have one of the following two possibilities.
\begin{itemize}
\item[(1)]
$c$ is isotopic to $d_{n-1}$ in $\Sigma_n\setminus\{x\}$.
\item[(2)]
There exists $g\in\Ker(\theta)$ such that $g([d_i])=[d_i]$ for all $1\le i\le n-1$ and $g([c])=[d_n]$.
\end{itemize}
\end{lem}

\begin{proof}
The surface $\Sigma_n$ is a surface of genus $\frac{n-2}{2}$ with two boundary components $\partial_1$ and $\partial_2$.
We assume that the circles $d_1,\dots,d_{n-1},d_n$ are arranged as in Figure \ref{fig5_4}.
Let $\Omega$ be the surface obtained by cutting $\Sigma_n$ along $\bigcup_{i=1}^{n-1}d_i$.
Then $\Omega$ has two connected components $\Omega_1$ and $\Omega_2$.
Each of these components is a cylinder that we represent by a square with a hole in the middle as shown in Figure \ref{fig5_5}.
Two opposite sides of each square represent arcs of $d_{n-2}$, one side represents an arc of $d_{n-1}$, and the last side represents a union of arcs of $d_1,\dots ,d_{n-3}$.
The boundary of the hole represents $\partial_1$ for $\Omega_1$ and $\partial_2$ for $\Omega_2$.
The puncture $x$ sits inside $\Omega_2$.
The trace of the circle $c$ in $\Omega$ is a simple arc $\ell$, either in $\Omega_1$ or in $\Omega_2$.

\begin{figure}[ht!]
\begin{center}
\includegraphics[width=5.2cm]{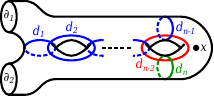}
\caption{Circles $d_1,\dots, d_n$}\label{fig5_4}
\end{center}
\end{figure}

\begin{figure}[ht!]
\begin{center}
\begin{tabular}{cc}
\parbox[c]{3.8cm}{\includegraphics[width=3.6cm]{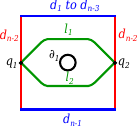}}
&
\parbox[c]{3.8cm}{\includegraphics[width=3.6cm]{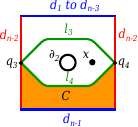}}\\
\\
$\Omega_1$ & $\Omega_2$
\end{tabular}
\caption{The surface $\Omega$}\label{fig5_5}
\end{center}
\end{figure}

Suppose $\ell$ is in $\Omega_1$.
Let $q$ be the intersection point of $c$ with $d_{n-2}$.
Then $q$ is represented in $\Omega_1$ by two points $q_1$ and $q_2$ on two opposite sides of $\Omega_1$ as shown in Figure \ref{fig5_5}, and $\ell$ is a simple arc connecting $q_1$ with $q_2$.
Up to isotopy pointwise fixing the boundary of $\Omega_1$, there exist exactly two simple arcs in $\Omega_1$ connecting $q_1$ to $q_2$ that are represented by the arcs $\ell_1$ and $\ell_2$ depicted in Figure \ref{fig5_5}.
The arc $\ell$ cannot be isotopic to $\ell_1$, otherwise $c$ would not be isotopic to $d_{n-1}$ in $\Sigma_n$.
So, $\ell$ is isotopic to $\ell_2$ in $\Omega_1$ which implies that $c$ is isotopic to $d_{n-1}$ in $\Sigma_n\setminus\{x\}$.

Now suppose $\ell$ is in $\Omega_2$.
Let $q$ be the intersection point of $c$ with $d_{n-2}$.
Then $q$ is represented in $\Omega_2$ by two points $q_3$ and $q_4$ on two opposite sides of $\Omega_2$ as shown in Figure \ref{fig5_5}, and $\ell$ is a simple arc connecting $q_3$ with $q_4$.
Up to isotopy (in $\Omega_2$ and not in $\Omega_2\setminus\{x\}$) pointwise fixing the boundary of $\Omega_2$, there exist exactly two simple arcs in $\Omega_2$ connecting $q_3$ to $q_4$ that are represented by the arcs $\ell_3$ and $\ell_4$ depicted in Figure \ref{fig5_5}.
The arc $\ell$ cannot be isotopic to $\ell_3$ in $\Omega_2$, otherwise $c$ would not be isotopic to $d_{n-1}$ in $\Sigma_n$.
So, $\ell$ is isotopic to $\ell_4$ in $\Omega_2$.
Let $\{F_t:\Omega_2\to \Omega_2\}_{t\in[0,1]}$ be an isotopy such that $F_0=\id$, $F_1(\ell)=\ell_4$ and $F_t $ is the identity on the boundary of $\Omega_2$ for all $t\in[0,1]$.
The arc $\ell_4$ divides $\Omega_2$ into two parts, the lower one which does not contain the hole bordered by $\partial_2$ and the puncture $x$, and the upper one which contains the hole bordered by $\partial_2$ and the puncture $x$, as shown in Figure \ref{fig5_5}.

Suppose $F_1(x)$ is in the upper part.
Let $C$ be the domain of $\Omega_2$ bounded by $\ell_4$, two arcs of $d_{n-2}$ and an arc of $d_{n-1}$ as shown in Figure \ref{fig5_5}.
Let $C'=F_1^{-1}(C)$.
Then $C'$ is a domain of $\Omega_2$ bounded by $\ell$, two arcs of $d_{n-2}$ and an arc of $d_{n-1}$ and $C'$  does not contain the puncture $x$.
The existence of such a domain implies that $c$ is isotopic to $d_{n-1}$ in $\Sigma_n\setminus\{x\}$.

Now, suppose $F_1(x)$ is in the lower part.
We can assume without loss of generality that the trace of $d_n$ on $\Omega_2$ is the simple arc $\ell_5$ drawn in Figure \ref{fig5_6}.
We can choose an isotopy $\{F_t':\Omega_2\to \Omega_2\}_{t\in[0,1]}$ such that $F_0'=\id$, $F_1'(\ell_4)= \ell_5$, $F_t'$ is the identity on the boundary of $\Omega_2$ for all $t\in [0,1]$, and $F_1'(F_1(x))=x$.
Let $\tilde F:\Sigma_n\to\Sigma_n$ be the homeomorphism which is $F_1'\circ F_1$ on $\Omega_2$ and is the identity outside $\Omega_2$, and let $g\in\MM(\Sigma_n,x)$ be the mapping class represented by $\tilde F$.
Then $g\in\Ker(\theta)$, $g([d_i])=[d_i]$ for all $1\le i\le n-1$, and $g([c])=[d_n]$.
\end{proof}

\begin{figure}[ht!]
\begin{center}
\includegraphics[width=3.6cm]{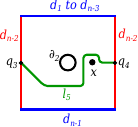}
\caption{The arc $\ell_5$}\label{fig5_6}
\end{center}
\end{figure}

\begin{rem}
The element $g$ at the end of the proof of Lemma \ref{lem5_5} is not necessarily trivial. 
For example, $\ell$ can be as shown in Figure \ref{fig5A2} up to isotopy and, in this case, $g$ must be non-trivial. 
In fact, $g$ can be any element of the fundamental group $\pi_1(\Omega_2,x)$, which is an infinite cyclic group, seen as a subgroup of $\MM (\Sigma_n,x)$.
\end{rem}

\begin{figure}[ht!]
\begin{center}
\includegraphics[width=3.6cm]{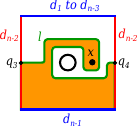}
\caption{An arc $\ell$ non-isotopic to $\ell_5$}\label{fig5A2}
\end{center}
\end{figure}

The following lemma is the extension of Lemma \ref{lem5_5} to the case $c \cap d_k \neq \emptyset$.

\begin{lem}\label{lem5_6}
Let $n$ be an even number, $n\ge4$.
Let $c$ be a generic circle of $\Sigma_n\setminus\{x\}$ such that $c\cap d_i=\emptyset$ for all $1\le i \le n-3$, $|c\cap d_{n-2}|=1$, and $c$ is isotopic to $d_{n-1}$ in $\Sigma_n$.
Then there exists $g\in\Ker(\theta)$ such that $g([d_i])=[d_i]$ for all $1\le i\le n-2$, and either $g([c]) =[d_{n-1}]$ or $g([c])=[d_n]$.
\end{lem}

\begin{proof}
In this proof the intersection number of two circles is computed on the surface $\Sigma_n \setminus \{ x\}$ and not on $\Sigma_n$.
We can assume that $|c\cap d_{n-1}|=i([c],[d_{n-1}])$ and $|c\cap d_n|=i([c],[d_n])$.
We argue by induction on $|c\cap d_{n-1}|+|c\cap d_n|=i([c],[d_{n-1}])+i([c],[d_n])$.
The case $|c\cap d_{n-1}|=0$ follows directly from Lemma \ref{lem5_5}, and the case $|c\cap d_n|=0$ is proved in the same way by replacing $d_{n-1}$ with $d_n$.
So we can assume that $i([c],[d_{n-1}])=|c\cap d_{n-1}|\ge 1$, $i([c],[d_n])=|c\cap d_n|\ge1$, and that the induction hypothesis holds.
Note that now $c$ and $d_{n-1}$ cannot be isotopic in $\Sigma_n\setminus\{x\}$.
Since $c$ and $d_{n-1}$ are isotopic in $\Sigma_n$, there exists a bigon $D$ in $\Sigma_n$ cobounded by an arc of $d_{n-1}$ and an arc of $c$ (see Figure \ref{fig5_7}).
Since $c$ and $d_{n-1}$ are not isotopic in $\Sigma_n\setminus\{x\}$, this bigon necessarily contains the puncture $x$.
We can choose $D$ to be minimal in the sense that its interior does not intersect $c$ and $d_{n-1}$.
Moreover, up to exchanging the roles of $d_{n-1}$ and $d_n$ if necessary, we can also assume that $d_n$ does not intersect the interior of $D$.
Clearly, $D$ does not intersect $d_i$ for any $1\le i\le n-3$ and, up to replacing $c$ with an isotopic circle, we can assume that $D$ does not intersect $d_{n-2}$ either.
Let $c'$ be a circle parallel to $c$ except in the bigon $D$ where it follows the arc of $d_{n-1}$ which borders $D$ as illustrated in Figure \ref{fig5_7}.
We have $c'\cap d_i=\emptyset$ for all $1\le i\le n-3$, $|c'\cap d_{n-2}|=1$ and $c'$ is isotopic to $d_{n-1}$ in $\Sigma_n$.
We also have $i([c'],[d_{n-1}])<i([c],[d_{n-1}])$ and $i([c'],[d_n])\le i([c],[d_n])$, hence by the induction hypothesis there exists $g_1\in\Ker(\theta)$ such that $g_1([d_i])=[d_i]$ for all $1\le i\le n-2$, and either $g_1([c'])=[d_{n-1}]$ or $g_1([c'])=[d_n]$.
Without loss of generality we can assume that $g_1([c'])=[d_{n-1}]$.
We choose $G_1\in\Homeo^+(\Sigma_n,x)$ which represents $g_1$ such that $G_1(d_i)=d_i$ for all $1\le i\le n-2$ and $G_1(c')=d_{n-1}$.
We set $c''=G_1(c)$.
Then $c''\cap d_i=\emptyset$ for all $1\le i\le n-3$, $|c''\cap d_{n-2}|=1$, $c''\cap d_{n-1}=\emptyset$, and $c''$ is isotopic to $d_{n-1}$ in $\Sigma_n$.
By Lemma \ref{lem5_5} there exists $g_2\in\Ker(\theta)$ such that $g_2([d_i])=[d_i]$ for all $1\le i\le n-2$, and either $g_2([c''])=[d_{n-1}]$ or $g_2([c''])=[d_n]$.
We set $g=g_2\circ g_1$.
Then $g\in\Ker(\theta)$, $g([d_i])=[d_i]$ for all $1\le i\le n-2$, and either $g([c])=[d_{n-1}]$ or $g([c])=[d_n]$.
\end{proof}

\begin{figure}[ht!]
\begin{center}
\includegraphics[width=3.6cm]{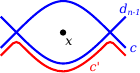}
\caption{Bigon cobounded by $c$ and $d_{n-1}$}\label{fig5_7}
\end{center}
\end{figure}

\begin{lem}\label{lem5_7}
 Let $n\ge6$.
Let $c_1,\dots,c_{n-1}$ be generic circles in $\Sigma_n\setminus\{x\}$ such that
\begin{itemize}
\item[(a)]
$|c_i\cap c_j|=1$ if $|i-j|=1$ and $|c_i\cap c_j|=0$ if $|i-j|\ge2$, for all $1\le i,j\le n-1$,
\item[(b)]
$c_i$ is isotopic to $d_i$ in $\Sigma_n$ for all $1\le i\le n-1$.
\end{itemize}
Then:
\begin{itemize}
\item[(1)]
If $n$ is odd, then there exists $g\in\Ker(\theta)$ such that $g([c_i])=[d_i]$ for all $1\le i\le n-1$.
\item[(2)]
If $n$ is even, then there exists $g\in\Ker(\theta)$ such that $g([c_i])=[d_i]$ for all $1\le i\le n-2$, and either $g([c_{n-1}])=[d_{n-1}]$ or $g([c_{n-1}])=[d_n]$.
\end{itemize}
\end{lem}

\begin{proof}
For $1\le k\le n-2$ we construct by induction on $k$ an element $g_k\in\Ker(\theta)$ such that $g_k([c_i])=[d_i]$ for all $1\le i\le k$.
Assume $k=1$.
Then, by Lemma \ref{lem5_4} applied to $k=1$, there exists $g_1\in\Ker(\theta)$ such that $g_1([c_1])=[d_1]$.
Suppose $2\le k\le n-1$ and $g_{k-1}$ is constructed.
We choose $G_{k-1}\in\Homeo^+(\Sigma_n,x)$ which represents $g_{k-1}$ and such that $G_{k-1}(c_i)=d_i$ for all $1\le i\le k-1$, and we set $c_k'=G_{k-1}(c_k)$.
Note that, since $g_{k-1}\in\Ker(\theta)$, the circle $c_k'$ is isotopic to $c_k$ in $\Sigma_n$.
Then, by Lemma \ref{lem5_4}, there exists $h_k\in\Ker(\theta)$ such that $h_k([d_i])=[d_i]$ for all $1\le i\le k-1$ and $h_k([c_k'])=[d_k]$.
We set $g_k=h_k\circ g_{k-1}$.
Then $g_k([c_i])=[d_i]$ for all $1\le i\le k$.
Note that when $n$ is odd we can extend the induction to $k=n-1$ and conclude the proof here by setting $g=g_{n-1}$.
The case where $n$ is even requires an extra argument.

Assume $n$ is even.
We choose $G_{n-2}\in\Homeo^+(\Sigma_n,x)$ which represents $g_{n-2}$ and such that $G_{n-2}(c_i)=d_i$ for all $1\le i\le n-2$, and we set $c_{n-1}'=G_{n-2}(c_{n-1})$.
Again, since $g_{n-2}\in\Ker(\theta)$, the circle $c_{n-1}'$ is isotopic to $c_{n-1}$ in $\Sigma_n$.
By Lemma \ref{lem5_6} there exists $h_{n-1}\in\Ker(\theta)$ such that $h_{n-1}([d_i])=[d_i]$ for all $1\le i\le n-2$, and either $h_{n-1}([c_{n-1}'])=[d_{n-1}]$ or $h_{n-1}([c_{n-1}'])=[d_n]$.
We set $g=h_{n-1}\circ g_{n-2}$.
Then $g([c_i])=[d_i]$ for all $1\le i\le n-2$, and either $g([c_{n-1}])=[d_{n-1}]$ or $g([c_{n-1}])=[d_n]$.
\end{proof}

\begin{proof}[Proof of Theorem \ref{thm2_2}]
Let $n \ge 6$ and let $\varphi:A[A_{n-1}]\to A[D_n]$ be a homomorphism.
Composing $\varphi$ with $\pi$, we get a homomorphism $\pi \circ \varphi : A [A_{n-1}] \to A [D_n] \to A [A_{n-1}]$.
We know by Theorem \ref{thm3_4} that we have one of the following possibilities.
\begin{itemize}
\item
$\pi\circ\varphi$ is cyclic.
\item
There exist $\psi\in\langle\bar\chi\rangle$ and $p\in\Z$ such that $\pi\circ\varphi$ is conjugate to $\psi\circ\bar\gamma_p$.
\end{itemize}
By Lemma \ref{lem5_1}, if $\pi\circ\varphi$ is cyclic, then $\varphi$ is cyclic.
So, we can assume that there exists $\psi\in\langle\bar\chi\rangle$ and $p\in\Z$ such that $\pi\circ\varphi$ is conjugate to $\psi\circ\bar\gamma_p$.
Up to conjugating and composing $\varphi$ on the left by $\chi$ if necessary, we can assume that $\pi\circ\varphi=\bar\gamma_p$, that is, $(\pi\circ\varphi)(s_i)=s_i\Delta_A^{2p}$, where $\Delta_A$ denotes the Garside element of $A[A_{n-1}]$.

Set $U=\rho_A(\Delta_A^2)$.
If $n$ is odd, then $U^2=T_{\partial}$, where $\partial$ is the boundary component of $\Sigma_n$, and if $n$ is even, then $U=T_{\partial_1}T_{\partial_2}$, where $\partial_1$ and $\partial_2$ are the two boundary components of $\Sigma_n$ (see Labruère--Paris \cite[Proposition 2.12]{LabPar1}).
In particular $U^2\in Z(\MM(\Sigma_n))$ in both cases.

By Theorem \ref{thm3_6} we know that there exist generic circles $c_1,\dots,c_{n-1}$ in $\Sigma_n\setminus\{x\}$, $\varepsilon\in\{\pm 1\}$ and $f_0\in\MM(\Sigma_n,x)$ such that
\begin{itemize}
\item[(a)]
$|c_i\cap c_j|=1$ if $|i-j|=1$ and $|c_i\cap c_j|=0$ if $|i-j|\ge2$, for all $1\le i,j\le n-1$,
\item[(b)]
$f_0$ commutes with $T_{c_i}$ for all $1\le i\le n-1$,
\item[(c)]
$(\rho_D\circ\varphi)(s_i)=T_{c_i}^{\varepsilon}f_0$ for all $1\le i\le n-1$.
\end{itemize}
For $1\le i\le n-1$ we denote by $b_i$ the circle in $\Sigma_n$ obtained by composing $c_i:\S^1\to \Sigma_n\setminus\{x\}$ with the embedding $\Sigma_n\setminus\{x\}\hookrightarrow\Sigma_n$.
In addition we set $g_0=\theta(f_0)$.
Then $(\theta\circ\rho_D\circ\varphi)(s_i)=T_{b_i}^\varepsilon g_0$ for all $1\le i\le n-1$.
Note that, since $\theta \circ \rho_D = \rho_A \circ \pi$ (see Diagram \ref{eq3_1}), we also have
$(\theta \circ \rho_D \circ \varphi) (s_i) = (\rho_A \circ \bar \gamma_p) (s_i) = \rho_A (s_i \Delta_A^{2p}) = T_{a_i} U^p$ for all $1 \le i \le n-1$, where the $a_i$'s are the circles depicted
in Figure \ref{fig3_1}.

{\it Claim.}
We have $\varepsilon=1$, $g_0=U^p$ and $b_i$ is isotopic to $a_i$ in $\Sigma_n$ for all $1\le i\le n-1$.

{\it Proof of the claim.}
Note that $g_0 = \theta (f_0)$ commutes with $T_{b_i} = \theta (T_{c_i})$ and $U = \rho_A (\Delta_A^2)$ commutes with $T_{a_i} = \rho_A (s_i)$, hence $T_{b_i}^{2 \varepsilon} g_0^2 = (T_{b_i}^\varepsilon g_0)^2 = (T_{a_i} U^p)^2 = T_{a_i}^2 U^{2p}$.
Since $g_0^2$ commutes with $T_{b_i}^{2\varepsilon}g_0^2=T_{a_i}^2U^{2p}$ and $U^2\in Z(\MM(\Sigma_n))$, $g_0^2$ commutes with $T_{a_i}^2$ for all $1\le i\le n-1$.
By Lemma \ref{lem5_2} it follows that $g_0^4\in Z(\MM(\Sigma_n))$.
By Proposition \ref{prop3_2} applied to $\MM (\Sigma_n)$ we deduce that $\SS(T_{a_i}^4U^{4p})=\SS(T_{b_i}^{4\varepsilon}g_0^4)=\{[a_i]\}=\{[b_i]\}$, hence $[a_i]=[b_i]$ for all $1\le i\le n-1$. 
Then $T_{a_i}^{4-4\varepsilon}=U^{-4p}g_0^4$, hence, by Proposition \ref{prop3_2}, $4-4\varepsilon=0$, and therefore $\varepsilon=1$.
Finally, from the equality $T_{a_i}U^p=T_{a_i}g_0$ it follows that $g_0=U^p$.
This finishes the proof of the claim.

From the claim it follows that $c_i$ is isotopic to $d_i$ in $\Sigma_n$ hence, by Lemma \ref{lem5_7}, there exists $g\in\Ker(\theta)$ such that $g([c_i]) =[d_i]$ for all $1\le i\le n-2$, $g([c_{n-1}])=[d_{n-1}]$ if $n$ is odd, and either $g ([c_{n-1}])=[d_{n-1}]$ or $g([c_{n-1}])=[d_n]$ if $n$ is even.
These equalities imply that $gT_{c_i}g^{-1}=T_{d_i}$ for $1\le i\le n-2$, $gT_{c_{n-1}}g^{-1}=T_{d_{n-1}}$ if $n$ is odd, and either $gT_{c_{n-1}}g^{-1}=T_{d_{n-1}}$ or $gT_{c_{n-1}}g^{-1}=T_{d_n}$ if $n$ is even.
By Theorem \ref{thm3_5}\,(1) there exists $v\in\Ker(\pi)$ such that $\rho_D(v)=g$.
So, up to composing $\varphi$ on the left by $\ad_v$ first, and composing on the left by $\zeta$ if necessary after, we can assume that $(\rho_D\circ\varphi)(s_i)=T_{d_i}f_0$ for all $1\le i\le n-1$, where $f_0$ commutes with $T_{d_i}$ for all $1\le i\le n-1$.
Since $T_{d_1}=\rho_D(t_1)\in\Im(\rho_D)$, we have $f_0\in\Im(\rho_D)$, hence there exists $u_0\in A[D_n]$ such that $\rho_D(u_0)=f_0$.
Since $\rho_D$ is injective (see Theorem \ref{thm3_5}), we deduce that $\varphi(s_i)=t_iu_0$ for all $1\le i\le n-1$ and $u_0$ commutes with $t_i$ for all $1\le i\le n-1$.
We set $Y=\{t_1,\dots,t_{n-1}\}$, $\Delta_Y=\Delta_Y[D_n]$, $\Delta_D=\Delta[D_n]$, $\kappa=2$ if $n$ is odd, and $\kappa=1$ if $n$ is even.
By Paris \cite[Theorem 1.1]{Paris1} the centralizer of $Y$ in $A[D_n]$ is generated by $\Delta_Y^2$ and $\Delta_D^\kappa$, hence there exists $q,r\in\Z$ such that $u_0=\Delta_Y^{2q}\Delta_D^{\kappa r}$.
We conclude that $\varphi=\beta_{q,r}$.
\end{proof}


\section{Endomorphisms of $A[D_n]$}\label{sec6}

The following lemma is a counterpart of Lemma \ref{lem5_7} for the case of odd $n$, and it is a preliminary to the proof of Theorem \ref{thm2_3}.

\begin{lem}\label{lem6_1}
Let $n$ be an odd number, $n\ge 5$.
Let $c$ be a generic circle of $\Sigma_n\setminus\{x\}$ such that $c\cap d_i=\emptyset$ for $1\le i \le n-3$, $|c\cap d_{n-2}|=1$, $c\cap d_{n-1}=\emptyset$, and $c$ is isotopic to $d_{n-1}$ in $\Sigma_n$.
Then we have one of the following three possibilities.
\begin{itemize}
\item[(1)]
$c$ is isotopic to $d_{n-1}$ in $\Sigma_n\setminus\{x\}$.
\item[(2)]
There exists $g\in\Ker(\theta)$ such that $g([d_i])=[d_i]$ for all $1\le i\le n-1$ and $g([c])=[d_n]$.
\item[(3)]
There exists $g\in\Ker(\theta)$ such that $g([d_i])=[d_i]$ for all $1\le i\le n-2$, $g([d_{n-1}])=[d_n]$ and $g([c])=[d_{n-1}]$.
\end{itemize}
\end{lem}

\begin{proof}
The surface $\Sigma_n$ is a surface of genus $\frac{n-1}{2}$ with one boundary component, $\partial$.
We assume that the circles $d_1,\dots,d_{n-1},d_n$ are arranged as shown in Figure \ref{fig6_1}.
The circles $d_{n-3}$ and $d_{n-1}$ divide $d_{n-2}$ into two arcs, $e_1$ and $e_2$, where the arc $e_1$ intersects $d_n$ and the arc $e_2$ does not intersect $d_n$ (see Figure \ref{fig6_1}).
Let $\Omega$ be the surface obtained by cutting $\Sigma_n$ along $\bigcup_{i=1}^{n-1}d_i$.
Then $\Omega$ is a cylinder represented by an octagon with a hole in the middle (see Figure \ref{fig6_2}).
Two opposite sides of this octagon represent arcs of $d_{n-1}$ and two opposite sides represent arcs of $d_1,\dots,d_{n-3}$, as shown in the figure.
Two other sides represent arcs of $e_1$ and the last two sides represent arcs of $e_2$, arranged as shown in Figure \ref{fig6_2}.
The boundary of the hole represents $\partial$.

\begin{figure}[ht!]
\begin{center}
\includegraphics[width=6.6cm]{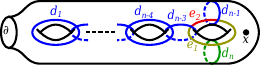}
\caption{The circles $d_1,\dots,d_n$}\label{fig6_1}
\end{center}
\end{figure}

\begin{figure}[ht!]
\begin{center}
\includegraphics[width=3.6cm]{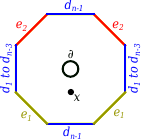}
\qquad
\includegraphics[width=3.6cm]{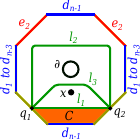}
\qquad
\includegraphics[width=3.6cm]{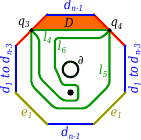}
\caption{The surface $\Omega$}\label{fig6_2}
\end{center}
\end{figure}

The circle $c$ intersects $d_{n-2}$ in a point $q$, and $q$ is either on the arc $e_1$ or on the arc $e_2$.
Suppose first that $q$ is on the arc $e_1$.
Then $q$ is represented on $\Omega$ by two points $q_1$ and $q_2$ lying on two different sides of $\Omega$ that represent $e_1$, and the trace of $c$ in $\Omega$ is a simple arc $\ell$ connecting $q_1$ to $q_2$.
Up to isotopy (in $\Omega$ and not in $\Omega\setminus\{x\}$) pointwise fixing the boundary of $\Omega$, there are exactly two simple arcs in $\Omega$ connecting $q_1$ to $q_2$ represented by the arcs $\ell_1$ and $\ell_2$ depicted in  Figure \ref{fig6_2}.
The arc $\ell$ cannot be isotopic to $\ell_2$, otherwise $c$ would not be isotopic to $d_{n-1}$ in $\Sigma_n$.
So, $\ell$ is isotopic to $\ell_1$ in $\Omega$.
Let $\{F_t:\Omega\to \Omega\}_{t\in[0,1]}$ be an isotopy such that $F_0=\id$, $F_1(\ell)=\ell_1$ and $F_t$ is the identity on the boundary of $\Omega$ for all $t\in [0,1]$.
The arc $\ell_1$ divides $\Omega$ into two parts, the lower one which does not contain the hole bounded by $\partial$ and the puncture $x$, and the upper one which contains the hole bounded by $\partial$ and the puncture $x$, as shown in Figure \ref{fig6_2}.

Suppose $F_1(x)$ is in the upper part.
Let $C$ be the domain of $\Omega$ bounded by $\ell_1$, two arcs of $e_1$ and an arc of $d_{n-1}$ as shown in Figure \ref{fig6_2}.
Let $C'=F_1^{-1}(C)$.
Then $C'$ is a domain of $\Omega$ bounded by $\ell$, two arcs of $e_1$ and an arc of $d_{n-1}$ which does not contain the puncture $x$.
The existence of such a domain implies that $c$ is isotopic to $d_{n-1}$ in $\Sigma_n\setminus\{x\}$.

Suppose $F_1(x)$ is in the lower part.
We can suppose that the trace of $d_n$ on $\Omega$ is the arc $\ell_3$ depicted in Figure \ref{fig6_2}.
We can choose an isotopy $\{F'_t:\Omega\to\Omega\}_{t\in[0,1]}$ such that $F'_0=\id$, $F_1'(\ell_1)= \ell_3$, $F_t'$ is the identity on the boundary of $\Omega$ for all $t\in [0,1]$, and $F_1'(F_1(x))=x$.
Let $\tilde F:\Sigma_n\to\Sigma_n$ be the homeomorphism which is $F_1'\circ F_1$ on $\Omega$ and is the identity outside $\Omega$, and let $g\in\MM(\Sigma_n,x)$ be the mapping class represented by $\tilde F$.
Then $g\in\Ker(\theta)$, $g([d_i])=[d_i]$ for all $1\le i\le n-1$, and $g([c])=[d_n]$.

Suppose now that $q$ is on the arc $e_2$.
Then $q$ is represented on $\Omega$ by two points $q_3$ and $q_4$ lying on two different sides of $\Omega$ which represent $e_2$, and the trace of $c$ in $\Omega$ is a simple arc $\ell$ connecting $q_3$ to $q_4$.
Up to isotopy (in $\Omega$ and not in $\Omega\setminus\{x\}$) pointwise fixing the boundary of $\Omega$, there are exactly two simple arcs in $\Omega$ connecting $q_3$ to $q_4$ represented by the arcs $\ell_4$ and $\ell_5$ depicted in Figure \ref{fig6_2}.
The arc $\ell$ cannot be isotopic to $\ell_5$, otherwise $c$ would not be isotopic to $d_{n-1}$ in $\Sigma_n$.
So, $\ell$ is isotopic to $\ell_4$ in $\Omega$.
Let $\{F_t:\Omega\to\Omega\}_{t\in[0,1]}$ be an isotopy such that $F_0=\id$, $F_1(\ell)=\ell_4$ and $F_t$ is the identity on the boundary of $\Omega$ for all $t\in [0,1]$.
The arc $\ell_4$ divides $\Omega$ into two parts, the upper one which does not contain the hole bounded by $\partial$ and the puncture $x$, and the lower one which contains the hole bounded by $\partial$ and the puncture $x$, as shown in Figure \ref{fig6_2}.

Suppose $F_1(x)$ is in the lower part.
Let $D$ be the domain of $\Omega$ bounded by $\ell_4$, two arcs of $e_2$ and an arc of $d_{n-1}$ as shown in Figure \ref{fig6_2}.
Let $D'=F_1^{-1}(D)$.
Then $D'$ is a domain of $\Omega$ bounded by $\ell$, two arcs of $e_2$ and an arc of $d_{n-1}$ which does not contain the puncture $x$. 
The existence of such a domain implies that $c$ is isotopic to $d_{n-1}$ in $\Sigma_n\setminus\{x\}$.

Suppose $F_1(x)$ is in the upper part.
Let $c'$ be the circle drawn in Figure \ref{fig6_3}.
We can assume that the trace of $c'$ on $\Omega$ is the arc $\ell_6$ drawn in Figure \ref{fig6_2}.
We can choose an isotopy $\{F'_t:\Omega\to\Omega\}_{t\in[0,1]}$ such that $F'_0=\id$, $F_1'(\ell_4)= \ell_6$, $F_t'$ is the identity on the boundary of $\Omega$ for all $t\in [0,1]$, and $F_1'(F_1(x))=x$.
Let $\tilde F:\Sigma_n\to\Sigma_n$ be the homeomorphism which is $F_1'\circ F_1$ on $\Omega$ and is the identity outside $\Omega$, and let $g_1\in\MM(\Sigma_n,x)$ be the mapping class represented by $\tilde F$.
Then $g_1\in\Ker(\theta)$, $g_1([d_i])=[d_i]$ for all $1\le i\le n-1$, and $g_1([c])=[c']$.

\begin{figure}[ht!]
\begin{center}
\includegraphics[width=7.6cm]{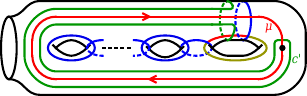}
\caption{The circle $c'$ and the loop $\mu$}\label{fig6_3}
\end{center}
\end{figure}

Let $g_2\in\pi_1(\Sigma_n,x)=\Ker(\theta)$ be the element represented by the loop $\mu$ drawn in Figure \ref{fig6_3}.
Let us mention here that $g_2$ is not the Dehn twist $T_\mu$ along $\mu$, but rather the image of the point-pushing map applied to $\mu$, which is equal to $T_{\mu_1} T_{\mu_2}^{-1}$ for $\mu_1$ and $\mu_2$ the two boundary curves of a small regular neighborhood of $\mu$, as explained in Farb--Margalit \cite[Section 4.2.2]{FarMar1}.
We have $g_2([d_i])=[d_i]$ for all $1\le i\le n-2$, $g_2([d_{n-1}])=[d_n]$ and $g_2([c'])=[d_{n-1}]$.
Set $g=g_2\circ g_1$.
Then $g\in\Ker(\theta)$, $g([d_i])=[d_i]$ for all $1\le i\le n-2$, $g([d_{n-1}])=[d_n]$ and $g([c])=[d_{n-1}]$.
\end{proof}

\begin{proof}[Proof of Theorem \ref{thm2_3}]
Let $n\ge6$.
Let $\varphi:A[D_n]\to A[D_n]$ be an endomorphism.
Consider the composition homomorphism $\varphi \circ \iota : A [A_{n-1}] \to A[D_n] \to A[D_n]$.
We know from Theorem \ref{thm2_2} that we have one of the following two possibilities up to conjugation.
\begin{itemize}
\item[(1)]
$\varphi\circ\iota$ is cyclic.
\item[(2)]
There exist $\psi\in\langle\zeta,\chi\rangle$ and $p,q\in\Z$ such that 
$\varphi\circ\iota=\psi\circ\beta_{p,q}$.
\end{itemize}

Suppose $\varphi\circ\iota$ is cyclic.
Then there exists $u\in A[D_n]$ such that $\varphi(t_i)=(\varphi\circ\iota)(s_i)=u$ for all $1\le i\le n-1$.
We also have
\begin{gather*}
\varphi(t_n)=\varphi(t_{n-2}t_nt_{n-2}t_n^{-1}t_{n-2}^{-1})=\varphi(t_{n-2}t_n )\,\varphi(t_{n-2})\,
\varphi(t_n^{-1}t_{n-2}^{-1})=\\
\varphi(t_{n-2}t_n)\,\varphi(t_1)\,\varphi(t_n^{-1}t_{n-2}^{-1})=
\varphi(t_1)=u\,,
\end{gather*}
hence $\varphi$ is cyclic.

So, we can assume that there exist $\psi\in\langle\zeta,\chi\rangle$ and $p,q\in\Z$ such that $\varphi\circ\iota$ is conjugate to $\psi\circ\beta_{p,q}$.
We set $Y=\{t_1,\dots,t_{n-2},t_{n-1}\}$, $\Delta_Y=\Delta_Y[D_n]$, $\Delta_D=\Delta[D_n]$, $\kappa=2$ if $n$ is odd, and $\kappa=1$ if $n$ is even.
Up to conjugating and composing $\varphi$ on the left by $\zeta$ if necessary, we can assume that there exist $\varepsilon\in\{\pm 1\}$ and $p,q\in\Z$ such that $\varphi(t_i)=(\varphi\circ\iota)(s_i)=t_i^\varepsilon\Delta_{Y}^{2p}\Delta_D^{\kappa q}$ for all $1\le i\le n-1$. 
The remainder of the proof is divided into four cases depending on whether $p$ is zero or not and whether $n$ is even or odd.

{\it Case 1: $n$ is even and $p\neq 0$.}
Then $\Sigma_n$ is a surface of genus $\frac{n-2}{2}$ with two boundary components, $\partial_1$ and $\partial_2$, and $\kappa=1$.
We have $\rho_D(t_i)=T_{d_i}$ for $1\le i\le n-1$ and, by Labruère--Paris \cite[Proposition 2.12]{LabPar1}, $\rho_D(\Delta_Y^2)=T_eT_{\partial_1}$ and $\rho_D(\Delta_D)=T_{\partial_1}T_{\partial_2}$, where $e$ is the circle drawn in Figure \ref{fig6_4}.
Set $f_i=(\rho_D\circ\varphi)(t_i)$ for all $1\le i\le n$.
Then, by the above,
\[
f_i=T_{d_i}^\varepsilon T_e^{p}T_{\partial_1}^{p+q}T_{\partial_2}^q\quad\text{for all }1\le i\le n-1\,.
\]
In particular, $\SS(f_i)=\{[d_i],[e]\}$ for all $1\le i\le n-1$.
Since $t_n$ is conjugate in $A[D_n]$ to $t_1$, $f_n$ is conjugate to $f_1$ in $\MM(\Sigma_n,x)$, hence $f_n$ is of the form $f_n=T_{d'}^\varepsilon T_{e'}^{p} T_{\partial_1}^{p+q}T_{\partial_2}^q$, where $d'$ is a non-separating circle and $e'$ is a circle that separates $\Sigma_n$ into two components, one being a cylinder containing $x$ and the other being a surface of genus $\frac{n-2}{2}$ with two boundary components, $\partial_1$ and $e'$, which does not contain $x$.
Moreover, by Theorem \ref{thm2_1}, $(\pi\circ\varphi)(t_{n-1})=(\pi\circ\varphi)(t_n)$, hence
\[
T_{d_{n-1}}^\varepsilon T_e^p T_{\partial_1}^{p+q} T_{\partial_2}^q = \theta (f_{n-1}) = \theta (f_n) = T_{d'}^\varepsilon T_{e'}^p T_{\partial_1}^{p+q} T_{\partial_2}^q 
\]
on $\Sigma_n$, that is, $T_{d_{n-1}}^\varepsilon T_e^p =  T_{d'}^\varepsilon T_{e'}^p$ as multitwists on $\Sigma_n$.
Now we can invoke Farb--Margalit \cite[Lemma 3.14]{FarMar1} to conclude that each curve of the set $\{ d_{n-1}, e \}$ is isotopic to a curve from the set $\{d', e'\}$ in $\Sigma_n$. 
To decide which curve of one set is isotopic to which curve in the other set we observe that removing a puncture does not change the property of a curve being non-separating, but can make a separating curve peripheral. 
Since both $d_{n-1}$ and $d'$ are non-separating, whereas $e$ and $e'$ are both separating or peripheral in $\Sigma_n$, we conclude that $d_{n-1}$ is isotopic to $d'$ in $\Sigma_n$ (and also that $e$ is isotopic to $e'$ in $\Sigma_n$).

We have $f_1f_n=f_nf_1$, hence by Theorem \ref{thm3_1}\,(3) we have $f_n(\SS(f_1))=\SS(f_1)$, thus $[e]$ is a reduction class for $f_n$, and therefore $i([e],[e'])=0$, because $[e']$ is an essential reduction class for $f_n$.
We can choose representatives $e$ and $e'$ such that $e \cap e' = \emptyset$ either by eliminating bigons, or by choosing geodesic representatives. 
Let $C, C' \subset \Sigma_n$ be cylinders containing $x$ and having boundary $\partial_2 \cup e$, $\partial_2 \cup e'$, respectively. 
Then either $C \subset C'$ if $e \subset C'$, or $C' \subset C$ if $e' \subset C$, with $x \in C \cap C'$. 
Say $C \subset C'$. 
Being a separating circle on $\Sigma_n$, $e$ separates $C'$ into two subsurfaces, one containing $\partial_2$ and $x$, and the other containing $e'$.
Being a subsurface with two boundary components lying inside a cylinder, the latter must be a cylinder itself. 
This cylinder establishes an isotopy between $e$ and $e'$ in $\Sigma_n \setminus \{ x\}$, hence $[e] = [e']$.
So, we can assume that $e=e'$.

Choose representatives $d_{n-1}$, $d'$ in minimal position in $\Sigma_n \setminus \{ x\}$.
Denote by $C_0$ and $\Sigma'$ the two components into which the curve $e$ separates $\Sigma_n$, with $C_0$ being a cylinder containing $x$, and $\Sigma'$ being the rest of the surface $\Sigma_n$, containing $d_1, \dots, d_{n-1}$. 
Suppose $d_{n-1} \cap d' \neq \emptyset$.
Then $d_{n-1}$ and $d'$ cobound a bigon. 
Since $d_{n-1}$ and $d'$ were chosen to be in minimal position in $\Sigma_n \setminus \{ x\}$, such a bigon must contain $x$. 
This implies that $d'$ has nonempty intersection with the cylinder $C_0$ which $e$ separates from the rest of the surface $\Sigma_n$, and since $e$ and $d'$ are disjoint, $d'$ lies entirely in $C_0$.
This is not possible because any generic circle in $C_0$ is peripheral in $\Sigma_n$ and $d'$ is non-separating in $\Sigma_n$.
So, $d_{n-1} \cap d' = \emptyset$.
Then there exists an embedded cylinder $C$ in $\Sigma_n$ with boundary components $d_{n-1}$ and $d'$.
Since $e$ is disjoint from $d'$ and $d_{n-1}$, $e$ either lies entirely in $C$ or is disjoint from $C$. 
The circle $e$ cannot lie entirely in $C$ because $e$ is peripheral in $\Sigma_n$ and, since both $d_{n-1}$ and $d'$ are non-separating in $\Sigma_n$, any generic circle lying in $C$ must be non-separating.
So, $e$ is disjoint from $C$, hence $C$ lies in $\Sigma'$, and therefore $d_{n-1}$ is isotopic to $d'$ in $\Sigma_n \setminus \{ x\}$.
Thus, we can also assume $d' = d_{n-1}$.

In conclusion we have $(\rho_D\circ\varphi)(t_{n-1})=(\rho_D\circ\varphi)(t_n)= T_{d_{n-1}}^\varepsilon T_e^{p}T_{\partial_1}^{p+q}T_{\partial_2}^q$, hence $\varphi(t_{n-1})=\varphi(t_n)=t_{n-1}^\varepsilon\Delta_Y^{2p}\Delta_D^q$.
We conclude that $\varphi=\beta_{p,q}\circ\pi$ if $\varepsilon=1$ and $\varphi=\chi\circ\beta_{-p,-q}\circ\pi$ if $\varepsilon=-1$.

\begin{figure}[ht!]
\begin{center}
\includegraphics[width=5.8cm]{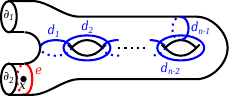}
\caption{Circles in $\Sigma_n$, case where $n$ is even and $p\neq 0$}\label{fig6_4}
\end{center}
\end{figure}

{\it Case 2: $n$ is odd and $p\neq 0$.}
Then $\Sigma_n$ is a surface of genus $\frac{n-1}{2}$ with one boundary component, $\partial$, and $\kappa=2$.
We have $\rho_D(t_i)=T_{d_i}$ for $1\le i\le n-1$ and, by Labruère--Paris \cite[Proposition 2.12]{LabPar1}, $\rho_D(\Delta_Y^4)=T_e$ and $\rho_D(\Delta_D^2)=T_{\partial}$, where $e$ is the circle drawn in Figure \ref{fig6_5}.
Set $f_i=(\rho_D\circ\varphi)(t_i)$ for all $1\le i\le n$.
Then, by the above,
\[
f_i^2=T_{d_i}^{2\varepsilon}T_e^{p}T_{\partial}^{2q}\quad\text{for all }1\le i\le n-1\,.
\]
In particular, $\SS(f_i)=\SS(f_i^2)=\{[d_i],[e]\}$ for all $1\le i\le n-1$.
The element $t_n$ is conjugate to $t_1$ in $A[D_n]$, hence $\varphi(t_n)$ is conjugate to $\varphi(t_1)$ in $A[D_n]$, and therefore there exists $ v\in A[D_n]$ such that $\varphi(t_n)=v\,\varphi(t_1)\,v^{-1}=(vt_1^\varepsilon v^{-1})(v\Delta_Y^{2p}v^{-1})\Delta_D^{2q}$.
The element $\rho_D(vt_1v^{-1})$ is conjugate to $\rho_D(t_1)=T_{d_1}$, hence $\rho_D(vt_1v^{-1})=T_{d'}$, where $d'$ is a non-separating circle.
The element $\rho_D(v\Delta_Y^4v^{-1})$ is conjugate to $\rho_D(\Delta_Y^4)=T_e$, hence $\rho_D(v\Delta_Y^4v^{-1})=T_{e'}$, where $e'$ is a circle that separates $\Sigma_n$ into two components, one being a cylinder containing $x$ and the other being a surface of genus $\frac{n-1}{2}$ with one boundary component which does not contain $x$.
We also have $f_n^2=T_{d'}^2T_{e'}^pT_\partial^{2q}$ and $\SS(f_n)=\SS(f_n^2)=\{[d'],[e']\}$.
By Theorem \ref{thm2_1} $(\pi\circ\varphi)(t_{n-1})=(\pi\circ\varphi)(t_n)$, hence $\theta(f_{n-1}^2)=\theta(f_n^2)$.
This implies that $d'$ is isotopic to $d_{n-1}$ in $\Sigma_n$.

Since $f_1f_n=f_nf_1$, by Theorem \ref{thm3_1}\,(3) we have $f_n^2(\SS(f_1))=\SS(f_1)$, hence $[e]$ is a reduction class for $f_n^2$, and therefore $i([e],[e'])=0$, because $[e']$ is an essential reduction class for $f_n^2$.
As in Case 1, we can choose representatives $e$ and $e'$ such that $e \cap e' = \emptyset$. 
Let $C, C' \subset \Sigma_n$ be cylinders containing $x$ and having boundary $\partial \cup e$, $\partial \cup e'$, respectively. 
Then either $C \subset C'$ if $e \subset C'$, or $C' \subset C$ if $e' \subset C$, with $x \in C \cap C'$. 
Say $C \subset C'$. 
Being a separating circle on $\Sigma_n$, $e$ separates $C'$ into two subsurfaces, one containing $\partial$ and $x$, and the other containing $e'$.
Being a subsurface with two boundary components lying inside a cylinder, the latter must be a cylinder itself. 
This cylinder establishes an isotopy between $e$ and $e'$ in $\Sigma_n \setminus \{ x\}$, hence $[e] = [e']$.
So, we can assume that $e = e'$, hence $\rho_D (v \Delta_Y^4 v^{-1}) = T_{e'} = T_e = \rho_D (\Delta_Y^4)$. Since $\rho_D$ is injective, it follows that $v \Delta_Y^4 v^{-1} = \Delta_Y^4$.

Using the same argument as in Case 1, from the fact that $d'$ does not intersect $e' = e$ and that $d'$ is isotopic to $d_{n-1}$ in $\Sigma_n$ it follows that $d'$ is isotopic to $d_{n-1}$ in $\Sigma_n \setminus \{ x \}$, hence we can also assume that $d' = d_{n-1}$.
Then $\rho_D (vt_1v^{-1}) = T_{d'} = T_{d_{n-1}} = \rho_D (t_{n-1})$, hence, since $\rho_D$ is injective, $vt_1v^{-1}=t_{n-1}$.
At this stage of the proof we have that $\varphi(t_n)=t_{n-1}^{\varepsilon}(v\Delta_Y^{2p}v^{-1})\Delta_D^{2q}$ and $(v\Delta_Y^{2p}v^{-1})^2=v\Delta_Y^{4p}v^{-1}=\Delta_Y^{4p}$.
It remains to show that $v\Delta_Y^{2p}v^{-1}=\Delta_Y^{2p}$.

\begin{figure}[ht!]
\begin{center}
\includegraphics[width=5.4cm]{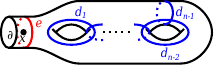}
\caption{Circles in $\Sigma_n$, case where $n$ is odd and $p\neq 0$}\label{fig6_5}
\end{center}
\end{figure}

By Theorem \ref{thm2_2} there exists $\psi\in\langle\zeta,\chi\rangle$ and $r,s\in\Z$ such that $\varphi\circ\zeta\circ\iota$ is conjugate to $\psi\circ\beta_{r,s}$.
The automorphism $\zeta$ is inner since $n$ is odd, hence we can assume that $\psi\in\langle\chi\rangle$.
So, there exist $w\in A[D_n]$, $\mu\in\{\pm 1\}$ and $r,s\in\Z$ such that $\varphi(t_i)=wt_i^\mu\Delta_Y^{2r}\Delta_D^{2s}w^{-1}$ for all $1\le i\le n-2$ and $\varphi(t_n)= wt_{n-1}^\mu\Delta_Y^{2r}\Delta_D^{2s}w^{-1}$.
Set $g=\rho_D(w)$.
We have $(\rho_D\circ\varphi)(t_i^2)=T_{d_i}^{2 \varepsilon}T_e^pT_{\partial}^{2q}=gT_{d_i}^{2 \mu}T_e^rT_{\partial}^{2s} g^{-1}$ for all $1\le i\le n-2$ and $(\rho_D\circ\varphi)(t_n^2)=T_{d_{n-1}}^{2 \varepsilon} T_e^p T_{\partial}^{2q} = g T_{d_{n-1}}^{2 \mu} T_e^rT_{\partial}^{2s}g^{-1}$.
So, $g^{-1}(\SS(T_{d_i}^{2 \varepsilon}T_e^pT_{\partial}^{2q}))=\SS(T_{d_{i}}^{2\mu}T_e^rT_{\partial}^{2s})$, hence $g^{-1}(\{[d_i],[e]\})\subset\{[d_i],[e]\}$ for all $1\le i\le n-1$.
This implies $g^{-1}([d_i])=[d_i]$, hence $g$ commutes with $T_{d_i}$, and therefore $w$ commutes with $t_i$ for all $1\le i \le n-1$.
Since $\Delta_Y$ is in the subgroup of $A[D_n]$ generated by $Y=\{t_1,\dots,t_{n-1}\}$ and $\Delta_D^2$ is central, it follows that $\varphi(t_i)=t_i^\mu\Delta_Y^{2r}\Delta_D^{2s}$ for all $1\le i\le n-2$ and $\varphi(t_n)=t_{n-1}^\mu\Delta_Y^{2r}\Delta_D^{2s}$.
Consider the equality $\varphi (t_1) = t_1^\varepsilon \Delta_Y^{2p} \Delta_D^{2q} = t_1^\mu \Delta_Y^{2r} \Delta_D^{2s}$. 
Then $t_1^{\varepsilon - \mu} \Delta_Y^{2 (p - r)} = \Delta_D^{2 (s - q)}$. 
The right-hand side of this equality lies in the center of $A[D_n]$, the left-hand side lies in $A_Y[D_n]$ and, by Paris \cite[Corollary 2.6]{Paris3}, the intersection of $A_Y [D_n]$ with the center of $A [D_n]$ is trivial, hence $s = q$ and $t_1^{\varepsilon - \mu} = \Delta_Y^{2 (r -p)}$. 
The element $\Delta_Y^{2 (r -p)}$ lies in the center of $A_Y [D_n]$ and $\langle t_1 \rangle$ is a proper parabolic subgroup of $A_Y [D_n]$, hence, again by Paris \cite[Corollary 2.6]{Paris3}, $t_1^{\varepsilon - \mu} = \Delta_Y^{2 (r -p)} = 1$, and therefore $\varepsilon = \mu$ and $r = p$. 
Here we use that $A[D_n]$ is torsion-free, which follows from Deligne \cite{Delig1} where it is proved that $A[D_n]$ has a finite-dimensional classifying space.
So, $\varphi(t_n)=t_{n-1}^\varepsilon\Delta_Y^{2p}\Delta_D^{2q}$.
We conclude that $\varphi=\beta_{p,q}\circ\pi$ if $\varepsilon=1$ and $\varphi=\chi\circ\beta_{-p,-q}\circ\pi$ if $\varepsilon=-1$.

{\it Case 3: $n$ is even and $p=0$.}
Then, again, $\Sigma_n$ is a surface of genus $\frac{n-2}{2}$ with two boundary components, $\partial_1$ and $\partial_2$, and $\kappa=1$.
We have $\rho_D(t_i)=T_{d_i}$ for $1\le i \le n-1$ and, by Labruère--Paris \cite[Proposition 2.12]{LabPar1}, $\rho_D(\Delta_D)= T_{\partial_1}T_{\partial_2}$.
Set $f_i=(\rho_D\circ\varphi)(t_i)$ for all $1\le i\le n$.
Then, by the above,
\[
f_i=T_{d_i}^\varepsilon T_{\partial_1}^{q}T_{\partial_2}^q\quad\text{for all }1\le i\le n-1\,.
\]
In particular, $\SS(f_i)=\{[d_i]\}$ for all $1\le i\le n-1$.
Since $t_n$ is conjugate in $A[D_n]$ to $t_1$, $f_n$ is of the form $f_n=T_{d'}^\varepsilon T_{\partial_1}^{q}T_{\partial_2} ^q$ where $d'$ is a non-separating circle.

For $1\le i\le n-3$ we have $t_it_n=t_nt_i$, hence $T_{d_i}T_{d'}=T_{d'}T_{d_i}$, and therefore, by Proposition \ref{prop3_0}, $i([d_i],[d'])=0$.
Similarly, we have $i([d_{n-1}],[d'])=0$.
Since $t_{n-2}t_nt_{n-2}=t_nt_{n-2}t_n$, we have $T_{d_{n-2}}T_{d'}T_{d_{n-2}} =T_{d'}T_{d_{n-2}}T_{d'}$, hence, by Proposition \ref{prop3_0}, $i([d_{n-2}],[d'])=1$.
So  we can assume that $d_i\cap d'=\emptyset$ for $1\le i\le n-3$, $d_{n-1}\cap d'=\emptyset$ and $|d_{n-2}\cap d'|=1$.
Moreover, by Theorem \ref{thm2_1}, $(\pi\circ\varphi)(t_{n-1})=(\pi\circ\varphi)(t_n)$, hence $\theta(f_{n-1})=\theta(f_n)$, and therefore $d'$ is isotopic to $d_{n-1}$ in $\Sigma_n$.
By Lemma \ref{lem5_5} it follows that we have one of the following two possibilities.
\begin{itemize}
\item[(1)]
$d'$ is isotopic to $d_{n-1}$ in $\Sigma_n\setminus\{x\}$.
\item[(2)]
There exists $g\in\Ker(\theta)$ such that $g([d_i])=[d_i]$ for all $1\le i\le n-1$ and $g([d'])=[d_n]$.
\end{itemize}

Suppose $d'$ is isotopic to $d_{n-1}$ in $\Sigma_n\setminus\{x\}$.
Then $(\rho_D\circ\varphi)(t_n)=T_{d_{n-1}}^\varepsilon T_{\partial_1}^{q}T_{\partial_2}^q$, hence, since $\rho_D$ is injective, $\varphi(t_n)=t_{n-1}^\varepsilon\Delta_D^q$.
We conclude that $\varphi=\beta_{0,q}\circ\pi$ if $\varepsilon=1$ and $\varphi=\chi\circ\beta_{0,-q}\circ\pi$ if $\varepsilon=-1$.

Suppose there exists $g\in\Ker(\theta)$ such that $g([d_i])=[d_i]$ for all $1\le i\le n-1$ and $g([d'])=[d_n]$.
We have
\[
(\rho_D\circ\varphi)(t_i)=
T_{d_i}^\varepsilon T_{\partial_1}^qT_{\partial_2}^q=
g^{-1}T_{d_i}^\varepsilon T_{\partial_1}^qT_{\partial_2}^qg
\]
for all $1\le i\le n-1$ and
\[
(\rho_D\circ\varphi)(t_n)=
T_{d'}^\varepsilon T_{\partial_1}^qT_{\partial_2}^q=
g^{-1}T_{d_n}^\varepsilon T_{\partial_1}^qT_{\partial_2}^qg\,.
\]
By Theorem \ref{thm3_5} there exists $v\in\Ker(\pi)\subset A[D_n]$ such that $\rho_D(v)=g$.
Since, $\rho_D$ is injective it follows that
\[
\varphi(t_i)=v^{-1}t_i^\varepsilon\Delta_D^qv\quad\text{for all }1\le i\le n\,.
\]
We conclude that $\varphi=\ad_{v^{-1}}\circ\gamma_q$ if $\varepsilon=1$ and $\varphi=\ad_{v^{-1}}\circ\chi\circ\gamma_{-q}$ if $\varepsilon=-1$.

{\it Case 4: $n$ is odd and $p=0$.}
Then, again, $\Sigma_n$ is a surface of genus $\frac{n-1}{2}$ with one boundary component, $\partial$, and $\kappa=2$.
We have $\rho_D(t_i)=T_{d_i}$ for $1\le i\le n-1$ and, by Labruère--Paris \cite[Proposition 2.12]{LabPar1}, $\rho_D(\Delta_D^2)=T_{\partial}$.
Set $f_i=(\rho_D\circ\varphi)(t_i)$ for all $1\le i\le n$.
Then, by the above,
\[
f_i=T_{d_i}^\varepsilon T_{\partial}^{q}\quad\text{for all }1\le i\le n-1\,.
\]
In particular, $\SS(f_i)=\{[d_i]\}$ for all $1\le i\le n-1$.
Since $t_n$ is conjugate in $A[D_n]$ to $t_1$, $f_n$ is conjugate to $f_1$ in $\MM(\Sigma_n,x)$, hence $f_n$ is of the form $f_n=T_{d'}^\varepsilon T_{\partial}^q$ where $d'$ is a non-separating circle.

For $1\le i\le n-3$ we have $t_it_n=t_nt_i$, hence $T_{d_i}T_{d'}=T_{d'}T_{d_i}$, and therefore, by Proposition \ref{prop3_0}, $i([d_i],[d'])=0$.
Similarly, we have $i([d_{n-1}],[d'])=0$.
Since $t_{n-2}t_nt_{n-2}=t_nt_{n-2}t_n$, we have $T_{d_{n-2}}T_{d'}T_{d_{n-2}}=T_{d'}T_{d_{n-2}}T_{d'}$, hence, by Proposition \ref{prop3_0}, $i([d_{n-2}],[d'])=1$.
So, we can assume that $d_i\cap d'=\emptyset$ for $1\le i \le n-3$, $d_{n-1}\cap d'=\emptyset$ and $|d_{n-2}\cap d'|=1$.
Moreover, by Theorem \ref{thm2_1}, $(\pi\circ\varphi)(t_{n-1})=(\pi\circ\varphi)(t_n)$, hence $\theta(f_{n-1})=\theta(f_n)$, and therefore $d'$ is isotopic to $d_{n-1}$ in $\Sigma_n$.
By Lemma \ref{lem6_1} it follows that we have one of the following three possibilities.
\begin{itemize}
\item[(1)]
$d'$ is isotopic to $d_{n-1}$ in $\Sigma_n\setminus\{x\}$.
\item[(2)]
There exists $g\in\Ker(\theta)$ such that $g([d_i])=[d_i]$ for all $1\le i\le n-1$ and $g([d'])=[d_n]$.
\item[(3)]
There exists $g\in\Ker(\theta)$ such that $g([d_i])=[d_i]$ for all $1\le i\le n-2$, $g([d_{n-1}])=[d_n]$ and $g([d'])=[d_{n-1}]$.
\end{itemize}

If $d'$ is isotopic to $d_{n-1}$ in $\Sigma_n\setminus\{x\}$, then we prove as in the case where $n$ is even that $\varphi=\beta_{0,q}\circ\pi$ if $\varepsilon=1$ and $\varphi=\chi\circ\beta_{0,-q}\circ\pi$ if $\varepsilon=-1$.
Similarly, if there exists $g\in\Ker(\theta)$ such that $g([d_i])=[d_i]$ for all $1\le i\le n-1$ and $g([d'])=[d_n]$, then we prove as in the case where $n$ is even that $\varphi=\ad_{v^{-1}}\circ\gamma_q$ if $\varepsilon=1$ and $\varphi=\ad_{v^{-1}}\circ\chi\circ\gamma_{-q}$ if $\varepsilon=-1$, where $v$ is an element of $\Ker(\pi)\subset A[D_n]$.

Suppose there exists $g\in\Ker(\theta)$ such that $g([d_i])=[d_i]$ for all $1\le i\le n-2$, $g([d_{n -1}])=[d_n]$ and $g([d'])=[d_{n-1}]$.
We have
\begin{gather*}
(\rho_D\circ\varphi)(t_i)=T_{d_i}^\varepsilon T_\partial^q=g^{-1}T_{d_i}^\varepsilon T_\partial^qg\quad\text{for }1\le i\le n-2\,,\\
(\rho_D\circ\varphi)(t_{n-1})=T_{d_{n-1}}^\varepsilon T_\partial^q=g^{-1}T_{d_n}^\varepsilon T_\partial^qg\,,\\
(\rho_D\circ\varphi)(t_n)=T_{d'}^\varepsilon T_\partial^q=g^{-1}T_{d_{n-1}}^\varepsilon T_\partial^qg\,.
\end{gather*}
By Theorem \ref{thm3_5} there exists $v\in\Ker(\pi)\subset A[D_n]$ such that $\rho_D(v)=g$.
Since $\rho_D$ is injective it follows that
\begin{gather*}
\varphi(t_i)=v^{-1}t_i^\varepsilon\Delta_D^{2q}v\quad\text{ for }1\le i\le n-2\,,\\
\varphi(t_{n-1})=v^{-1}t_n^\varepsilon\Delta_D^{2q}v\,,\quad\varphi(t_n)=v^{-1}
t_{n-1}^\varepsilon\Delta_D^{2q}v\,.
\end{gather*}
We conclude that $\varphi=\ad_{v^{-1}}\circ\zeta\circ\gamma_q$ if $\varepsilon=1$ and 
$\varphi=\ad_{v^{-1}}\circ\zeta\circ\chi\circ\gamma_{-q}$ if $\varepsilon=-1$.
\end{proof}


\section{Endomorphisms of $A [D_n] / Z (A [D_n])$}\label{sec7}

\begin{proof}[Proof of Proposition \ref{propA1}]
Let $\Delta$ be the Garside element of $A[D_n]$. 
We set $\kappa = 2$ if $n$ is odd and $\kappa = 1$ if $n$ is even. 
Recall that $Z (A [D_n])$ is the cyclic group generated by $\Delta^{\kappa}$. 
Let $\varphi_Z : A_Z [D_n] \to A_Z [D_n]$ be an endomorphism. 
For each $1 \le i \le n-2$ we define $u_i \in A[D_n]$ by induction on $i$ as follows. 
First, choose any $u_1 \in A[D_n]$ such that $\xi (u_1) = \varphi_Z (t_{Z,1})$. 
Now, assume that $2 \le i \le n-2$ and that $u_{i-1}$ is defined. 
Choose $u_i' \in A [D_n]$ such that $\xi (u_i') = \varphi_Z (t_{Z,i})$. 
Since $\varphi_Z (t_{Z,i-1} t_{Z,i} t_{Z,i-1}) = \varphi_Z (t_{Z, i} t_{Z, i-1} t_{Z, i})$, there exists $k_i \in \Z$ such that $u_{i-1} u_i' u_{i-1} = u_i' u_{i-1} u_i' \Delta^{\kappa k_i}$. 
Then set $u_i = u_i' \Delta^{\kappa k_i}$. 
Note that $\xi (u_i) = \xi (u_i') = \varphi_Z (t_{Z,i})$ and
\[
u_{i-1} u_i u_{i-1} = u_{i-1} u_i' u_{i-1} \Delta^{\kappa k_i} = u_i' u_{i-1} u_i' \Delta^{2 \kappa k_i} = u_i u_{i-1} u_i \,.
\]
Define in the same way $u_{n-1}, u_n \in A[D_n]$ such that $\xi (u_{n-1}) = \varphi_Z (t_{Z, n-1})$, $\xi (u_n) = \varphi_Z (t_{Z, n})$, $u_{n-2} u_{n-1} u_{n-2} = u_{n-1} u_{n-2} u_{n-1}$ and $u_{n-2} u_n u_{n-2} = u_n u_{n-2} u_n$.

Let $i,j \in \{1, \dots, n\}$ be such that $i \neq j$ and $t_i t_j = t_j t_i$. 
We have $\varphi_Z (t_{Z, i} t_{Z, j}) = \varphi_Z (t_{Z, j} t_{Z, i})$, hence there exists $\ell \in \Z$ such that $u_i u_j = u_j u_i \Delta^{\kappa \ell}$. 
Recall the homomorphism $z : A [D_n] \to \Z$ which sends $t_i$ to $1$ for all $1 \le i \le n$. 
Since $z(\Delta) = n (n-1)$, the previous equality implies that
\[
z(u_i) + z(u_j) = z(u_j) + z(u_i) + \kappa \ell n (n-1)\,,
\]
hence $\ell = 0$, and therefore $u_i u_j = u_j u_i$.

By the above we have an endomorphism $\varphi : A [D_n] \to A [D_n]$ which sends $t_i$ to $u_i$ for all $1 \le i \le n$, and this endomorphism is a lift of $\varphi_Z$.
\end{proof}

\begin{proof}[Proof of Theorem \ref{thmA2}]
Let $n \ge 6$. 
Let $\varphi_Z : A_Z [D_n] \to A_Z [D_n]$ be an endomorphism. 
We know from Proposition \ref{propA1} that $\varphi_Z$ admits a lift $\varphi : A [D_n] \to A [D_n]$. 
By Theorem \ref{thm2_3} we have one of the following three possibilities up to conjugation.
\begin{itemize} 
\item[(1)]
$\varphi$ is cyclic.
\item[(2)]
There exist $\psi \in \langle \zeta, \chi \rangle$ and $p, q \in \Z$ such that $\varphi = \psi \circ \beta_{p, q} \circ \pi$.
\item[(3)]
There exist $\psi \in \langle \zeta, \chi \rangle$ and $p \in \Z$ such that $\varphi = \psi \circ \gamma_p$.
\end{itemize}
Clearly, if $\varphi$ is cyclic, then $\varphi_Z$ is cyclic.

Now, we show that the second case cannot occur. 
Suppose there exist $\psi \in \langle \zeta, \chi \rangle$ and $p, q \in \Z$ such that $\varphi = \psi \circ \beta_{p, q} \circ \pi$. 
As ever, we set $\kappa = 2$ if $n$ is odd and $\kappa = 1$ if $n$ is even.
Recall that the center of $A[D_n]$ is generated by $\Delta^\kappa$, where $\Delta$ is the Garside element of $A[D_n]$.
We need to show that $\varphi (\Delta^\kappa) \not\in Z (A[D_n]) = \langle \Delta^\kappa \rangle$, which leads to a contradiction.
Since $\psi \in \Aut (A [D_n])$, we have $\psi (Z (A [D_n])) = Z (A [D_n])$, hence we can assume that $\varphi = \beta_{p, q} \circ \pi$. 
Let $Y = \{t_1, \dots, t_{n-1}\}$ and let $\Delta_Y = \Delta_Y [D_n]$ be the Garside element of $A_Y [D_n]$.
Since  
\begin{gather*}
\Delta = (t_1 \cdots t_{n-2} t_{n-1} t_n t_{n-2} \cdots t_1) \cdots (t_{n-2} t_{n-1} t_n t_{n-2}) (t_{n-1} t_n)\,,\\
\Delta [A_{n-1}]^2 = (s_1 \cdots s_{n-2} s_{n-1}^2 s_{n-2} \cdots s_1) \cdots (s_{n-2} s_{n-1}^2 s_{n-2}) s_{n-1}^2
\end{gather*}
(see Lemma \ref{lemA5} for the second equality), we have $\pi (\Delta) = \Delta [A_{n-1}]^2$, hence
\[
\varphi (\Delta^\kappa) = (\beta_{p, q} \circ \pi) (\Delta^\kappa) = \beta_{p, q} (\Delta [A_{n-1}]^{2 \kappa}) = \Delta_Y^{2 \kappa (1+p n (n-1))} \Delta^{\kappa^2 q n (n-1)}\,.
\]
This element does not belong to $Z (A[D_n]) = \langle \Delta^\kappa \rangle$, because $\kappa (1 + p n (n-1)) \neq 0$ and $\langle \Delta_Y^2 \rangle \cap \langle \Delta^\kappa \rangle = \{ 1 \}$.

Suppose we are in the third case. 
So, there exist $\psi \in \langle \zeta, \chi \rangle$ and $p \in \Z$ such that $\varphi = \psi \circ \gamma_p$. 
We have
\[
\gamma_p (\Delta^\kappa) = \Delta^{\kappa (1+ \kappa p n (n-1))} \in \langle \Delta^\kappa \rangle\,,
\]
hence $\gamma_p$ induces an endomorphism $\gamma_{Z,p} : A_Z [D_n] \to A_Z [D_n]$.
Moreover, for all $1 \le i \le n$,
\[
\gamma_{Z,p} (t_{Z,i}) = \xi (t_i \Delta^{\kappa p}) = \xi (t_i) = t_{Z,i}\,,
\]
hence $\gamma_{Z,p} = \id$.
Clearly, $\psi$ is the lift of an element $\psi_Z \in \langle \zeta_Z, \chi_Z \rangle$, hence $\varphi_Z = \psi_Z \circ \gamma_{Z, p} = \psi_Z$.
\end{proof}

Now, as promised in Section \ref{sec2}, we prove the following.

\begin{lem}\label{lemA6}
Let $n \ge 4$. 
If $n$ is even, then $\langle \zeta_Z, \chi_Z \rangle \cap \Inn (A_Z [D_n]) = \{ \id \}$, and if $n$ is odd, then $\langle \chi_Z \rangle \cap \Inn (A_Z [D_n]) = \{ \id \}$.
\end{lem}

\begin{proof}
We first show that, if $\varphi : A [D_n] \to A[D_n]$ is an automorphism such that $\varphi_Z \in \Inn (A_Z [D_n])$, then $\varphi \in \Inn (A [D_n])$. 
Let $\varphi \in \Aut (A[D_n])$ be such that $\varphi_Z \in \Inn (A_Z [D_n])$. 
There exists $g_Z \in A_Z [D_n]$ such that $\varphi_Z (t_{Z,i}) = g_Z t_{Z,i} g_Z^{-1}$ for all $1 \le i \le n$. 
Again, we denote by $\Delta$ the Garside element of $A[D_n]$, and we set $\kappa = 2$ if $n$ is odd and $\kappa = 1$ if $n$ is even. 
Let $g \in A [D_n]$ be such that $\xi (g) = g_Z$. 
For every $1 \le i \le n$, there exists $k_i \in \Z$ such that $\varphi (t_i) = g t_i g^{-1} \Delta^{\kappa k_i}$. 
Let $i, j \in \{1, \dots, n\}$ be such that $\{t_i, t_j\}$ is an edge of $D_n$. 
From the equality $t_i t_j t_i = t_j t_i t_j$ it follows that
\[
g t_i t_j t_i g^{-1} \Delta^{\kappa (2k_i + k_j)} = \varphi (t_i t_j t_i) = \varphi (t_j t_i t_j) = g t_j t_i t_j g^{-1} \Delta^{\kappa (k_i + 2 k_j)} \,,
\]
hence $2 k_i + k_j = k_i + 2 k_j$, and therefore $k_i = k_j$. 
Since $D_n$ is a connected graph, it follows that $k_i = k_j$ for all $i,j \in \{1, \dots, n\}$. 
So, there exists $k \in \Z$ such that $\varphi (t_i) = g t_i g^{-1} \Delta^{\kappa k}$ for all $1 \le i \le n$. 
Recall the homomorphism $z : A [D_n] \to \Z$ which sends $t_i$ to $1$ for all $1 \le i \le n$. 
Since $\varphi$ is an automorphism, we have $\Im (z \circ \varphi) = \Im (z) = \Z$. 
Furthermore, since $z (\Delta) = n (n-1)$, we have $(z \circ \varphi) (t_i) = 1 + \kappa k n (n-1)$ for all $1 \le i \le n$, hence $\Im (z \circ \varphi) = \big( 1 + \kappa k n (n-1) \big) \Z$. 
This implies that $k = 0$, hence $\varphi = \ad_g \in \Inn (A [D_n])$.

Arguing in a similar way we can see that lifts of $\zeta_Z$ and $\chi_Z$ in $\Aut (A [D_n])$ are unique.
Since we know that $\langle \zeta, \chi \rangle \cap \Inn (A[D_n]) = \{ \id \}$ if $n$ is even, and $\langle \chi \rangle \cap \Inn (A[D_n]) = \{ \id \}$ if $n$ is odd, it follows that 
$\langle \zeta_Z, \chi_Z \rangle \cap \Inn (A_Z [D_n]) = \{ \id \}$ if $n$ is even, and $\langle \chi_Z \rangle \cap \Inn (A_Z [D_n]) = \{ \id \}$ if $n$ is odd.
\end{proof}



\end{document}